\documentclass[dvips,12pt]{article}

\usepackage{paralist}
\usepackage{graphicx}
\usepackage[a4paper, twoside, top=4cm, bottom=4cm, inner=2cm , outer=2cm]{geometry}
\usepackage[english]{babel}
\usepackage{amsmath, amsopn, amstext, amscd, amsfonts, amssymb,
  amsbsy, mathrsfs}
\usepackage{enumerate}
\usepackage{url}
\usepackage[numbers]{natbib}
\usepackage{varioref}
\usepackage{dsfont}
\usepackage{empheq}
\usepackage{tikz}
\usetikzlibrary{calc,patterns,arrows,decorations}

\usepackage{float}

\usepackage{amsthm}

\usepackage{hyperref}
\hypersetup{colorlinks=true, linkcolor=blue, urlcolor=blue, citecolor=blue, breaklinks=true}
\usepackage{url}
\usepackage{breakurl}

\newtheoremstyle{rem}{}{}{\itshape}{}{\large\scshape\bfseries}{.}{5pt}{}
\theoremstyle{rem}
\theoremstyle{rem}\newtheorem{theorem}{Theorem}[section]
\theoremstyle{rem}\newtheorem{lemma}{Lemma}[section]
\theoremstyle{rem}\newtheorem{proposition}{Proposition}[section]
\theoremstyle{rem}
\theoremstyle{rem}\newtheorem{remark}{Remark}[section]
\theoremstyle{rem}
\theoremstyle{rem}\newtheorem{corollary}{Corollary}[section]
\newtheoremstyle{ack}{}{}{}{}{\large\scshape\bfseries}{.}{5pt}{}
\theoremstyle{ack}\newtheorem*{acknowledgement}{Acknowledgement}

\newcommand\mycite[2][]{\citep[#1]{#2}}
\labelformat{section}{Section~#1}   
\labelformat{subsection}{Section~#1}   
\labelformat{subsubsection}{Section~#1}   
\labelformat{equation}{eq.~(#1)}
\labelformat{theorem}{Theorem~#1}
\labelformat{lemma}{Lemma~#1}
\labelformat{proposition}{Proposition~#1}
\labelformat{corollary}{Corollary~#1}
\labelformat{definition}{Definition~#1}
\labelformat{align}{eq.~(#1)}
\labelformat{remark}{Remark~#1}
\labelformat{figure}{Figure~#1}
\labelformat{table}{Table~#1}
\labelformat{example}{Example~#1}

\newcommand{\bee}{\begin{enumerate}}
\newcommand{\eee}{\end{enumerate}}

\newcommand{\floor}[1]{\lfloor #1 \rfloor}
\newcommand{\ceil}[1]{\lceil #1 \rceil}

\newcommand{\Ps}{ \mathcal{P}}

\newcommand{\Lip}{\mathscr{L}}

\newcommand{\Gr}{{\mathcal G}}

\newcommand{\A}{{\mathcal A}}
\newcommand{\As}{{\mathscr A}}
\newcommand{\m}{{\mathfrak m}}
\newcommand{\fdel}{{\mathfrak S}}

\newcommand{\J}{{\mathcal J}}

\newcommand{\V}{{\mathcal V}}

\newcommand{\Ma }{{\mathcal M}}

\newcommand{\D}{\mathcal{D}}
\newcommand{\beqn}{\begin{equation}}
\newcommand{\eeqn}{\end{equation}}

\newcommand{\norm}[1]{\|#1\|}

\let\epsilon=\varepsilon

\let\phi=\varphi

\let\tilde=\widetilde

\newcommand{ \ms }[2]{\langle #1,#2 \rangle}
\newcommand{ \abs }[1]{ \vert #1 \vert }

\newcommand{ \Zr }{\mathbb{Z}}

\newcommand{\rr }{R}
\newcommand{\Xs }{\mathcal{E}}
\newcommand{\Y }{\mathcal{Y}}
\newcommand{\Cs}{\mathscr{C}}
\newcommand{\Cd}{\mathcal{C}}
\newcommand{ \Bo }{\mathcal{B}}

\newcommand{ \Lp }{\mathbb{L}}

\newcommand{ \R }{\mathbb{R}}
\newcommand{ \N }{\mathbb{N}}
\newcommand{ \Os }{\mathcal{O}}
\newcommand{ \Sp }{\mathbb{S}}
\newcommand{ \Supp }{\text{Supp}}
\newcommand{ \Span }{\text{Span}}
\newcommand{ \Cl }{\text{Closure}}

\newcommand{ \Se }{\mathcal{S}}
\newcommand{ \Fu }{\mathcal{R}}

\newcommand{ \Hs }{ \mathcal{H} }
\newcommand{ \info}{\mathcal{F}}

\newcommand{ \Pb}{\mathbb{P}}

\newcommand{ \Exp }{ \mathbb{E} }
\newcommand{ \Var }{ \mathbb{V}\mathrm{ar} }

\newcommand{ \normL }[2]{ \Vert #1  \Vert_{#2} }

\newcommand{\ind}[1]{\mathds{1}_{ \{ #1 \} }}
\newcommand{\inds}{\mathds{1}}

\newcommand{\jn}{ {j,\nu} }

\newcommand{\jk}{ {j,k} }

\newcommand{\pr}{ \mathscr{P} }
\newcommand{\re}{ \mathscr{R} }
\newcommand{\ct}{ \mathscr{T} }
\newcommand{\lo}{ \mathscr{S} }

\newcommand{\Hp}[1]{$\mathbf{(#1)}$}
\newcommand{\mbf}[1]{\mathbf{#1}}

\begin{document}

\title{Classification via local multi-resolution projections}


\date{\today}
\author{Jean-Baptiste Monnier, \\ e-mail: \url{monnier@math.jussieu.fr}\\
Universit\'{e} Paris Diderot, Paris 7,\\ LPMA, office 5B01, \\ 175 rue du Chevaleret,\\ 75013, Paris,
  France.}

\maketitle

\begin{abstract}
We focus on the supervised binary classification problem, which
consists in guessing the label $Y$ associated to a co-variate $X \in \R^d$,
given a set of $n$ independent and identically distributed co-variates
and associated labels $(X_i,Y_i)$. We assume that the law of the
random vector $(X,Y)$ is unknown and the marginal law of $X$ admits a
density supported on a set $\A$. In the particular case of plug-in
classifiers, solving the classification problem boils down to the
estimation of the regression function $\eta(X) = \Exp[Y|X]$. Assuming
first $\A$ to be known, we show how it is possible to construct an
estimator of $\eta$ by localized projections onto a multi-resolution
analysis (MRA). In a second step, we show how this estimation
procedure generalizes to the case where $\A$ is
unknown. Interestingly, this novel estimation procedure presents
similar theoretical performances as the celebrated local-polynomial
estimator (LPE). In addition, it benefits from the lattice structure
of the underlying MRA and thus outperforms the LPE from a
computational standpoint, which turns out to be a crucial feature in many practical
applications. Finally, we prove that the associated plug-in classifier
can reach super-fast rates under a margin assumption.
\end{abstract}

\textsc{AMS 2000 subject classifications:} Primary 62G05, 62G08;
Secondary 62H30, 62H12.\\

\textsc{Key-Words:} Nonparametric regression; Random design;
Multi-resolution analysis; Supervised binary classification; Margin
assumption.


\section{Introduction}\label{sec:intro}

\subsection{Setting}

The supervised binary classification problem is directly related to a
wide range of applications such as spam detection or assisted medical
diagnosis (see \mycite[chap.~1]{Hastie2001} for more details). It can be described as follows.\\

\textit{The supervised binary classification problem.} Let $\Xs$ stand
for a subset of $\R^d$ and write $\Y = \{0,1\}$. Assume we observe $n$
co-variates $X_i\in \Xs$ and associated labels $Y_i \in \Y$ such that
the elements of $\D_n = \{ (X_i,Y_i), i=1,\ldots,n\}$ are $n$ independent realizations of
the random vector $(X,Y) \in \Xs \times \Y$ of unknown law
$\Pb_{X,Y}$. Given $\D_n$ and a new co-variate $X_{n+1}$, we want to
predict the associated label $Y_{n+1}$ so as to minimize the probability of making a mistake.\\

In other words, we want to build a \textbf{classifier} $h_n: \Xs
\mapsto \Y$ upon the data $\D_n$, which minimizes $\Pb(h_n(X) \neq Y
\vert \D_n)$. It is well known that the Bayes classifier $h^*(\tau) :=
\ind{\eta(\tau) \geq 1/2}$, where $\eta(\tau):= \Exp[Y \vert X = \tau]
= \Pb(Y=1 \vert X=\tau)$ (unknown in practice), is optimal among all
classifiers since, for any other classifier $h_n$, we have
$\ell(h_n,h^*):= \Pb(h_n(X) \neq Y\vert \D_n) - \Pb(h^*(X) \neq Y)
\geq 0$ (see \mycite{Devroye1996}). As a consequence, we measure the
classification risk $\ct(h_n)$ associated to a classifier $h_n$ as its
average relative performance over all data sets $\D_n$, $\ct(h_n) =
\Exp^{\otimes n} \ell(h_n,h^*)$. As described in
\mycite[Chap.~7]{Devroye1996}, there is no classifier $h_n$ such that
$\ct(h_n)$ goes to zero with $n$ at a specified rate for all
distributions $\Pb_{X,Y}$. We therefore make the assumption that $\Pb_{X,Y}$
belongs to a class of distributions $\Ps$ (as large as possible) and
aim at constructing a classifier $h_n$ such that
\begin{align}\label{eq:ClassifMinimaxDef}
\inf_{\theta_n} \sup_{\Pb_{X,Y} \in \Ps} \ct(\theta_n) &\lesssim
\sup_{\Pb_{X,Y} \in \Ps} \ct(h_n) \lesssim (\log n)^{\delta}
\inf_{\theta_n} \sup_{\Pb_{X,Y} \in \Ps} \ct(\theta_n), & n\geq 1,
\end{align}
where the infinimum is taken over all measurable maps $\theta_n$ from
$\Xs$ into $\Y$ and $\lesssim$ means lesser or equal up to a multiplicative constant
factor independent of $n$. Any classifier $h_n$ verifying \ref{eq:ClassifMinimaxDef} will be
said to be \textbf{(nearly) minimax optimal} when $\delta=0$
($\delta>0$). $\Ps$ will stand for the set of all
distributions such that the marginal law $\Pb_X$ of $X$ admits a
density $\mu$ on $\Xs$ and $\eta$ belongs to a given smoothness
class. Throughout the paper, we will denote by $\mu$ the density of $\Pb_X$.\\

Many classifiers have been suggested in the literature, such as
$k$-nearest neighbors, neural networks, support vector machine (SVM) or
decision trees (see \mycite{Devroye1996, Hastie2001}). In this paper, we will exclusively focus on
\textbf{plug-in classifiers} $h_n(\tau):= \ind{\eta_n(\tau) \geq 1/2}$, where
$\eta_n$ stands for an estimator of $\eta$. With such classifiers, it
is shown in \mycite{Vapnik1998} that,
\begin{align}\label{eq:plugintoreg}
\ct( h_n)\leq 2 \Exp^{\otimes n} \Exp \abs{\eta_n(X) - \eta(X)},
\end{align}
where the term on the rhs is known as the regression loss (of the
estimator $\eta_n$ of $\eta$) in $\Lp_1(\Xs,\mu)$-norm. \Ref{eq:plugintoreg} shows in
particular that rates of convergence on the classification risk of a
plug-in classifier $h_n$ can be readily derived from rates of
convergence on the regression loss of $\eta_n$. This prompts us to
focus on the regression problem, which can be stated in full generality
as follows.\\

\textit{The regression on a random design problem.} Let $\Xs,\Y$ stand
for subsets of $\R^d$ and $\R$, respectively. Assume we dispose of $n$
co-variates $X_i\in \Xs$ and associated observations $Y_i \in \Y$ such that
the elements of $\D_n = \{ (X_i,Y_i), i=1,\ldots,n\}$ are $n$ independent realizations of
the random vector $(X,Y) \in \Xs \times \Y$ of unknown law
$\Pb_{X,Y}$. We define $\xi := Y - \eta(X)$, where $\eta(\tau):=
\Exp[Y \vert X=\tau]$, so that by construction $\Exp[\xi\vert X] =
0$. Given $\D_n$ and under the assumption that $\Pb_{X,Y}$ belongs
to a large class of distributions $\Ps$, we want to
come up with an estimator $\eta_n$ of $\eta$, which is as accurate as
possible for the wide range of losses $\lo_p(\eta_n) =  \Exp^{\otimes n}
\Exp \abs{\eta_n(X) - \eta(X)}^p$, $p\geq1$.\\

As described previously, in the particular case where $\Y=\{0,1\}$, we fall back
on the regression problem associated to the classification problem
with plug-in classifiers. In this case, $\xi$ is bounded such that
$\abs{\xi}\leq 1$. Notice however that the regression
on a random design problem stated above permits for $\Y$ to be any
subset of $\R$ (including $\R$ itself). To be more precise, and by analogy with
\ref{eq:ClassifMinimaxDef}, our aim is to build an estimator $\eta_n$ of $\eta$ such
that, for all $p\geq1$,
\begin{align}\label{eq:RegMinimaxDef}
\inf_{\theta_n} \sup_{\Pb_{X,Y} \in \Ps} \lo_p(\theta_n) &\lesssim
\sup_{\Pb_{X,Y} \in \Ps} \lo_p(\eta_n) \lesssim (\log n)^{\delta}
\inf_{\theta_n} \sup_{\Pb_{X,Y} \in \Ps} \lo_p(\theta_n), & n\geq 1,
\end{align}
where the infinimum is taken over all measurable maps $\theta_n$ from
$\Xs$ into $\Y$. And $\eta_n$ will be said to be (nearly) minimax
optimal when $\delta=0$ ($\delta>0$). 

\subsection{Motivations}
Many estimators $\eta_n$ of $\eta$ have been suggested in the
literature to solve the regression on a random design problem. Among
them, the celebrated local polynomial
estimator (LPE) has been praised for its flexibility and strong
theoretical performances (see \mycite{Stone1980,Stone1982}). As is
well known, the LPE is minimax optimal in any dimension $d \in \N$ and
for any $\lo_p$-loss, $p \in (0,\infty]$, over the set of laws $\Ps$
such that \begin{inparaenum}[(i)]\item
  $\mu$ is bounded from above and below on its support $\A := \Supp
  \mu = \{\tau: \mu(\tau)>0 \}$, \item $\eta$ belongs to a H\"older
  ball $\Cs^s(\Xs, M)$ of radius $M$ and \item $\xi$ has
  sub-Gaussian tails\end{inparaenum}. As a drawback, the LPE
is computationally expansive since it requires to perform a new regression at
every single point $x\in \A$ where we want to estimate $\eta$.\\
Computational efficiency is however of primary importance in many practical
applications. In this paper, we show that it is possible to construct a novel
estimator $\eta_n$ of $\eta$ by localized projections onto multi-resolution analysis (MRA) of
$\Lp_2(\R^d,\lambda)$ (where $\lambda$ stands for the Lebesgue measure
on $\Xs$), which presents similar theoretical performances and is
computationally more efficient than the LPE.

\subsection{The hypotheses}\label{sec:hyp}
In this section, we summarize the assumptions on $\mu$, $\A$, $\eta$ and $\xi$
that will be used throughout the paper.\\

\textit{Assumption on $\mu$.}
Let us denote by $\mu_{\min}, \mu_{\max}$ two real numbers such that $0 < \mu_{\min}
\leq \mu_{\max} < \infty$. As is standard in the regression on a random design setting, we assume that the
density $\mu$ is bounded above and below on its support $\A$.  
\begin{compactdesc}
\item[\Hp{D1}] $\mu_{\min} \leq \mu(\tau) \leq \mu_{\max}$ for all
  $\tau \in \A$.
\end{compactdesc}
This guarantees that we have enough information at each point $x
\in \A$ in order to estimate $\eta$ with best accuracy. For a study with weaker assumptions on
$\mu$, the reader is referred to \mycite{Gaiffas2005, Gaiffas2007a}, for example, and the references
therein.\\

\textit{Assumption on $\A$.} 
We first assume that,
\begin{compactdesc}
\item[\Hp{S1}] $\A = \Xs = [0,1]^d$. 
\end{compactdesc}
Therefore $\A$ is known under \Hp{S1}. We will deal with the case
where $\A$ is unknown in \ref{sec:relaxS1}.\\

\textit{Assumption on $\eta$.}
Fix $r \in \N$. In the sequel, we will assume
that,
\begin{compactdesc}
\item[\Hp{H_s^r}] The regression function $\eta$ belongs to the
  generalized Lipschitz ball $\Lip^s(\Xs,M)$ of radius $M$, for some $s \in (0,r)$. 
\end{compactdesc}
Unless otherwise sated, $s$ is \textbf{unknown} but belongs
to the interval $(0,r)$, where $r$ is \textbf{known}. For a
detailed review of generalized Lipschitz classes, the reader is
referred to the Appendix below.\\

\textit{Assumptions on the noise $\xi$.}
We will consider the two following assumptions,
\begin{compactdesc}
\item[\Hp{N1}] Conditionally on $X$, the noise $\xi$ is uniformly
  bounded, meaning that there exists an absolute constant $K>0$ such that $\abs{\xi}\leq K$.
\item[\Hp{N2}] The noise $\xi$ is independent of $X$ and normally
  distributed with mean zero and variance $\sigma^2$, which we will denote by $\xi \sim \Phi(0,\sigma^2)$.
\end{compactdesc}
Assumption \Hp{N1} is adapted to the supervised binary classification
setting, where $\Y = \{0,1\}$, while \Hp{N2} is more common in the regression on a random
design setting, where $\Y = \R$.\\

\textit{Combination of assumptions.}
In the sequel, we will conveniently refer by \Hp{CS1} to the set of assumptions
\Hp{D1}, \Hp{S1}, \Hp{N1} or \Hp{N2}. As detailed below in
\ref{sec:review}, configuration \Hp{CS1} is comparable to what is customary in the
regression on a random design setting.


\section{Our results}
Assuming at first $\A$ to be known, we
introduce a novel nonparametric estimator $\eta^{@}$ of
$\eta$ built upon local regressions against a multi-resolution
analysis (MRA) of $\Lp_2(\R^d,\lambda)$ and show that, under \Hp{CS1},
it is adaptive nearly minimax optimal over a wide
generalized Lipschitz scale and across the wide range of
losses $\Lp_p(\Xs, \mu), p \in [1,\infty)$. We subsequently show that these
results generalize to the case where $\A$ is unknown but
belongs to a large class of (eventually disconnected) subsets of
$\R^d$, provided we modify the estimator $\eta^{@}$ accordingly. We
denote by $\eta^{\maltese}$ this latter estimator and prove that $\eta^{\maltese}$ can
be used to build an adaptive nearly minimax optimal plug-in classifier, which can reach super-fast
rates under a margin assumption. The above results essentially hinge on
an exponential upper-bound on the probability of deviation of $\eta^@$ from
$\eta$ at a point, as detailed in \ref{th:main}. These results either improve on the current
literature or are interesting in their own right for the following
reasons.
\begin{compactenum}[1)]
\item They show that it is possible to use MRAs to construct
  an adaptive nearly minimax optimal estimator $\eta^{@}$ of $\eta$ under the
  sole set of assumptions \Hp{CS1}. More precisely, our results \begin{inparaenum}[(i)]\item hold in any
  dimension $d$; \item over the wide range of $\Lp_p(\Xs,
  \mu)$-losses, $p \in [1,\infty)$; \item and a large
  Lipschitz scale; \item and do not require any assumption on $\mu$
  beyond \Hp{D1}. \end{inparaenum} It is noteworthy that, in contrary to most alternative
MRA-based estimation methods, no smoothness assumption on $\mu$ is
  needed.
\item From a computational perspective, $\eta^@$ outperforms other
  estimators of $\eta$ under \Hp{D1} since it takes
  full advantage of the lattice structure of the underlying MRA. In particular it
  requires at most as many regressions as there are data points to be
  computed everywhere on $\Xs$, while alternative kernel estimators must be recomputed at each single
  point of $\Xs$. We illustrate this latter feature through simulation.
\item Furthermore, and in contrary to alternative MRA-based estimators, the
  local nature of $\eta^{@}$ allows to relax the assumption
  that $\A$ is known. This latter configuration allows
  for $\mu$ to cancel on $\Xs$ as long as it remains bounded on its
  support $\A$, which is particularly appropriate to the supervised binary
  classification problem under a margin assumption. 
\item In the regression on a random design
  setting, $\eta^@$ bridges in fact the gap
  between usual linear wavelet
  estimators and alternative kernel estimators, such as the LPE. On the one hand, $\eta^@$ inherits its
  computational efficiency from the lattice structure of the
  underlying MRA. On the other hand, it features similar theoretical performances as the LPE in the random design
  setting. In particular, it remains a (locally) linear
  estimator of the data (modulo a spectral
  thresholding of the local regression matrix), and cannot
  discriminate finer smoothness than the one described by (generalized) Lipschitz spaces.  
\end{compactenum}

Here is the paper layout. We start by a literature review in
\ref{sec:review}. We give a hand-waving introduction to the
main ideas that underpin the local multi-resolution estimation
procedure in \ref{sec:handwaving}. We define
notations that will be used throughout the paper and introduce MRAs in
\ref{sec:assumptions}. Our actual estimation procedure is described in
\ref{sec:estproc} and the results are detailed in \ref{sec:res}. We show how these results can be
fine-tuned under additional assumptions in
\ref{sec:refinment}. Assumption \Hp{S1} is relaxed and the properties
of $\eta^{\maltese}$ are detailed in \ref{sec:relaxS1}. We show how these latter
results spread to the classification setting in
\ref{sec:classifres}. Results of a simulation study with $\eta^@$ under \Hp{CS1}
are given in \ref{sec:simuls}. Proofs of the regression results can be
found in \ref{sec:proofglobal}. The proofs of the classification results are simple
modifications of the proofs given in \mycite{Audibert2007} and can be
found in \mycite{Monnier2011}. In addition, the Appendix contains a
detailed review of generalized Lipschitz spaces and MRAs.

\section{Literature review}\label{sec:review}
Both the regression on a random design problem and the
classification problem have a long-standing
history in nonparametric statistics. We will
therefore limit ourselves to a brief account of the corresponding
literature that is relevant to the present paper.

\subsection{Classification with plug-in classifiers}
Let us start with a review of some of the classification
literature dedicated to plug-in classifiers. The seminal work
\mycite{Marron1983} showed that plug-in rules are asymptotically
optimal. It has been subsequently pointed out in \mycite{Mammen1999} that
the classification problem is in fact only sensitive
to the behavior of $\Pb_{X,Y}$ near the boundary line $\mathscr{M}:=\{\tau \in \Xs:
\eta(\tau) = 1/2 \}$. So that assumptions on the behavior of $\Pb_{X,Y}$
away from this boundary are in fact unnecessary. Subsequent works
such as \mycite{Audibert2004} have shown that convex
combinations of plug-in classifiers can reach fast rates (meaning
faster than $n^{-1/2}$, and thus faster than
nonparametric estimation rates). More recently, it has been shown in \mycite{Audibert2007} that plug-in
classifiers can reach super fast rates (that is faster than $n^{-1}$) under suitable
conditions. All these results are derived under some sort of smoothness assumption on the regression
function $\eta$ (see \mycite{Yang1999}) and a margin assumption
\Hp{MA} (see \ref{sec:classifres} for details). This latter assumption
clarifies the behavior of $\Pb_{X,Y}$ in a neighborhood of $\mathscr
M$ and kicks in naturally through the computation
\begin{align*}
\ct(h_n) &\leq \delta \Pb( 0 < \abs{2\eta(X) -1} \leq \delta)
+ \Exp \abs{ \eta_n(X) - \eta(X)} \ind{\abs{\eta_n(X) - \eta(X)}> \delta},
\end{align*}
where $\delta$ is chosen such that it balances the two
terms on the rhs. Finally, \mycite{Audibert2007} exhibited optimal convergence rates under
smoothness and margin assumptions and showed that they are attained
with plug-in classifiers. Let us now turn to the regression on a
random design problem.

\subsection{Regression on a random design with wavelets}
First results on multi-resolution analysis (MRA) and wavelet bases (see
\mycite{Mallat1989, Meyer1992}) emerged in the nonparametric statistics literature in the early
$1990$'s (see \mycite{Kerkyacharian1992, Donoho1994, Donoho1995,
  Donoho1995a, Donoho1996}). It has been proved that, under \Hp{CS1}
and in the particular case where $\mu$ is the uniform distribution on $\Xs$, thresholded wavelet
estimators of $\eta$ are nearly minimax optimal over a wide Besov
scale and range of $\Lp_p(\Xs, \mu)$-losses (see
\mycite{Delyon1996}). In order to leverage on the power of MRAs and
associated wavelet bases, several authors attempted to transpose these
latter results to more general design densities $\mu$. This, however,
led to a considerable amount of difficulties.\\
The literature relative to the study of wavelet estimators on an \textbf{unknown}
random design breaks down into two main
streams. \begin{inparaenum}[(i)] \item The first one aims at constructing
  new wavelet bases adapted to the (empirical) measure of the
  design (see \mycite{Kohler2003, Kohler2008, Delouille2001,
    Sweldens1996}). \item The second one aims at coming up with new
  algorithms to estimate the coefficients of the expansion of $\eta$
  on traditional wavelet bases (see \mycite{Antoniadis1997, Hall1997,
    Kovac2000, Neumann1995, Sardy1999})\end{inparaenum}. The present
paper belongs to this second line of research.\\
As described in \mycite{Hall1997}, the success of the LPE on a random
design results from the fact that it is built as a ``ratio'', which
cancels out most of the influence of the design. In a wavelet context,
a first suggestion has therefore been to use the ratio estimator of $\eta$ (see \mycite{Antoniadis1998,
  Pensky2001}, for example), well known from the statistics literature on
orthogonal series decomposition (see \mycite{Greblicki1982, Greblicki1985}
and \mycite[Chap.~17]{Devroye1996} and the references
therein). Roughly speaking, the ratio estimator is the wavelet
equivalent of the Nadaraya-Watson estimator (see \mycite{Nadaraya1964,
Watson1964}). It is elaborated on the simple observation that $\eta(x)
= \eta(x)\mu(x)/\mu(x)$ for all $x\in\A$, where both $g(.) =
\eta(.)\mu(.)$ and $\mu(.)$ are easily estimated via traditional
wavelet methods. The ratio estimator relies thus unfortunately on the estimation of
$\mu$ itself and must therefore assume as much smoothness on $\mu$ as
on $\eta$.\\
To address that issue, an other approach has been introduced in
\mycite{Cai1998, Kerkyacharian2004}. They work with $d=1$ and take $\Xs$
to be the unit interval $[0,1]$. Their approach relies on the wavelet
estimation of $\eta \circ G^{-1}$, where $G$ stands for the cumulative
distribution of the design and $G^{-1}$ for its generalized
inverse. Results are therefore stated in term of regularity of $f
\circ G^{-1}$. Unfortunately, this method does not readily generalize
to the the multi-dimensional case, where $G$ admits no inverse.\\
Finally, \mycite{Baraud2002} obtains adaptive near-minimax optimal
wavelet estimators over a wide Besov scale under \Hp{CS1} by means of model
selection techniques. His results are hence valid for the $\Lp_2(\Xs,
\mu)$-loss only.\\
Other relevant references that proceed with hybrid estimators (LPE and
kernel estimator or LPE and wavelet estimator) are 
\mycite{Gaiffas2007} and \mycite{Zhang2002}. They both work under \Hp{CS1}, with
$d=1$ and assume that $\mu$ is at least continuous.

\section{A primer on local multi-resolution estimation under
  \Hp{CS1}}\label{sec:handwaving}
\begin{figure}[htb]
\begin{center}
\begin{tikzpicture}
\draw[thin,color=gray] (-5,-2) grid (5,5);
\draw[->] (-5,0)--(5,0) node[right] {$\R$};

\foreach \num in {-5,...,5}
\fill [red] (\num, 0) circle (2pt);

\node (o) at (0,0) [below] {$0$};

\draw[color= blue] (0.5,-0.5) node[below] {$\Hs = 2^{-j}[0,1]$};

\foreach \num in {0,1,2,3,4}{
  \draw (-4 + \num, 0) .. controls +(2,2) and +(-3,2) .. (1+\num,0);
}

\node (v) at (-4,4) [above] {$V_j$};
\node (g) at (3,4) [above] {$d=1$};
\node (p1) at (-2,2) {$\phi_{j,k_1}$};
\node (p5) at (2,2) {$\phi_{j,k_5}$};
\draw[color = red] (4,-0.5) node[below] {$2^{-j}\Zr$};

\draw[line width = 2pt, color=blue] (0,0) -- (1,0);

\end{tikzpicture}
\end{center}
\caption{Description of the localization cells $\Hs$ and their
  relations to the $\Supp \phi_\jk$.}\label{figure:supportH}
\end{figure}

In order to fix the ideas, let us now give a hand-waving
introduction to the local multi-resolution estimation method.
Throughout the paper, we will work with $r$-MRAs of $\Lp_2(\R^d, \lambda)$, for some $r\in\N$,
consisting of nested approximation spaces $\V_{j}\subset
\V_{j+1}$ built upon compactly supported scaling
functions (see \ref{subsec:polrep} and Appendix). Under the assumption that $\eta$
belongs to the generalized Lipschitz ball $\Lip^s(\Xs,M)$ of radius $M$, the essential supremum of the remainder of
the orthogonal projection $\pr_j\eta$ of $\eta$ onto $\V_j$ decreases like
$2^{-js}$ (see Appendix). The regression function $\eta$ can therefore be legitimately approximated by
$\pr_j\eta$. As an element of $\V_j$, $\pr_j\eta$ may be written as an
infinite linear combination of scaling functions at level $j$. In
particular, there exists a partition $\info_j$ of $\Xs$ into
hypercubes of edge-length $2^{-j}$ such that, for all $\Hs \in
\info_j$ and all $x \in \Hs$, we can write $\pr_j\eta(x) = \sum_{k \in \Se_j(\Hs)}
\alpha_\jk \phi_\jk(x)$, where $\Se_j(\Hs)$ stands for a finite subset
of $\Zr^d$ (see \ref{figure:supportH}). This leaves us in turn with the estimation
of coefficients $(\alpha_\jk)_{k\in \Se_j(\Hs)}$ for all $\Hs \in
\info_j$, which is achieved by least-squares and provides us with the
estimator $\eta_j^@$ of $\eta$ on $\Hs$. It is noteworthy that the
local estimator $\eta_j^@$ of $\eta$ is exclusively built upon scaling
functions and does not require the estimation of wavelet
coefficients. In particular, it does not involve any sort of wavelet
coefficient thresholding. To the best of the author knowledge, this is the first
time that this local estimation procedure is proposed and studied from
both a theoretical and computational perspective. In addition, we show that
Lepski's method (see \mycite{Lepski1997}, for example) can be used to adaptively
choose the resolution level $j$. Notice that Lepski's method has already been used
in a MRA setting in \mycite{Picard2000}. In what follows, we detail
the local multi-resolution estimation method and establish the near minimax
optimality of $\eta^@$.

\section{Notations}\label{sec:assumptions}
\subsection{Preliminary notations}

In the sequel, we will denote by $\Bo_p(z,\rho)$ the closed
$\ell_p$-ball of $\R^d$ of center $z$ and radius $\rho$. More generally,
we adopt the following notations: for any subset $\Se$ of a topological space $\Xs$,
$\Cl(\Se)$ will stand for its closure and $\Se^c$ for its complement in
$\Xs$. For any subset $\Se$ of $\R^d$, $z \in \R^d$ and $\tau \in \R^+$, we will write $z + \Se$
and $\tau \Se$ to mean the sets $\{ z + u : u \in \Se\}$ and
$\{ \tau u: u \in \Se \}$, respectively. Finally, given a set (of functions) $\Fu$,
$\Span \Fu$ will denote the set of finite linear combinations of elements of
$\Fu$. \\
For any $p \in \N$, vectors $v$ of $\R^p$ will be seen as elements of
$\Ma_{p,1}$, that is matrix with $p$ rows and one column. For any two
$u,v \in \R^p$, $\ms{u}{v}$ will denote their Euclidean scalar
product. In addition, for any $p,q \in \N$ and $M \in \Ma_{p,q}$,
$M^t$ will stand for the transpose of $M$. For any two
matrices $M,P$, $M \cdot P$ will denote their matrix product when
it makes sense. $[M]_{k,\ell}$ and $[M]_{k,\bullet}$ will
respectively stand for the element of $M$ located at line $k$, column
$\ell$ and the $k^{th}$ row of $M$. Finally, $\normL{M}{S}$ will
denote the spectral norm of $M$ (see \mycite[\textsection 5.6.6]{Horn1990}).\\
We denote by $\floor{z}$ the integer part of $z \in \R$ defined as
$\max\{ a \in \Zr: a \leq z \}$. More generally, given $z \in \R^d$, we write $\floor{z}$ the
integer part of $z$, meant in a coordinate-wise sense. In the
same way, we denote by $\ceil{z}$ the smallest integer greater than
$z$ (in a coordinate-wise sense). We write rhs (resp. lhs) to mean \textsl{right-} (resp. \textsl{left-})
\textsl{hand-side} and sometimes write $:=$ to mean \textsl{equal by definition}. Throughout the
paper, we will refer to constants independent of $n$ as
\textsl{absolute constants} and $c,C$ will stand for absolute constants
whose value may vary from line to line. For any two sequences $a_n,b_n$ of $n$, we will write
$a_n \lesssim b_n$ to mean $a_n \leq Cb_n$ for some absolute constant
$C$ and $a_n \approx b_n$ to mean that there exist two constants $c,C$
independent of $n$ such that $cb_n \leq a_n \leq Cb_n$. 

\subsection{The polynomial reproduction property}\label{subsec:polrep}

In what follows, we will exclusively consider MRAs built upon
Daubechies' scaling functions $\phi_\jk$ (see Appendix and
\mycite{Daubechies1992, Mallat2008, Cohen2003, Hardle1997}). Given a natural integer $r$, we
will refer by $r$-MRA to a MRA whose nested approximation spaces $\V_j$ reproduce
polynomials up to order $r-1$. Daubechies' scaling functions $\phi_\jk$ are
appealing in the estimation framework since they
are compactly supported and have minimal volume supports among scaling
functions that give rise to $r$-MRAs. Recall finally that a $r$-MRA can
explain Lipschitz smoothness $s$ for any $s\in(0,r)$. 

\subsection{General notations}\label{generalnotations}
Consider the Daubechies' $r$-MRA of $\Lp_2(\R^d, \lambda)$ built upon
Daubechies' scaling function $\phi$, as described in the Appendix. We will denote by
$\Supp \phi_\jk = \{\tau \in \R^d: \phi_\jk(\tau) > 0\}$ the support
of $\phi_\jk$. Recall that $\Supp \phi = [-(r-1), r]^d$. To
alleviate notations, we will write $\phi_k$ in place of
$\phi_{0,k}$ and $\phi_j$ in place of
$\phi_{j,0}$. Notice that $\Cl(\Supp \phi_\jk)$ is in fact a closed
hyper-cube of $\R^d$ whose corners lie on the lattice
$2^{-j}\Zr^d$. For any $x \in \A$, we write 
\begin{align*}
\Se_j(x) &= \{\nu \in \Zr^d: x \in \Supp \phi_\jn \}.
\end{align*}
Furthermore, we write $\info_j := 2^{-j}((0,1)^d + \Zr^d) \cap \Xs$. It defines a
partition of $\Xs$ into $2^{jd}$ hypercubes of edge length
$2^{-j}$, modulo a $\lambda$-null set. For the sake of concision, we write $ \rr = 2r-1$ in the
sequel. We have the following proposition, whose proof is
straightforward and thus left to the reader.

\begin{proposition}\label{proposition:techSj}
$\Se_j$ verifies the following properties, 
\begin{compactenum}
\item $\Se_j$ is constant on each element $\Hs \in
\info_j$. We will denote by $\Se_j(\Hs)$ its value on
$\Hs$. 
\item Moreover, for any two $\Hs_1,
\Hs_2 \in \info_j$, $\Hs_1 \neq \Hs_2$, $\Se_j(\Hs_1)$ differs from $\Se_j(\Hs_2)$ by at
least one element.
\item Finally, for any $\Hs \in \info_j$, $\#\Se_j(\Hs) = \rr^d $ 
\end{compactenum}
\end{proposition}

It is a direct consequence of \ref{proposition:techSj} that in the
case where $r=1$, we have $\#\Se_j(\Hs) = 1$ for all $\Hs \in
\info_j$. We denote its single element by
$\nu(\Hs)$. It is in fact easy to show that $\nu(\Hs) =
\floor{2^jx}$ for any $x \in \Hs$. For any $\Hs \in \info_j$, we write
\begin{align*}
\alpha_{\Hs} &= (
\alpha_\jn)_{\nu  \in  \Se_j(\Hs)} \in
\R^{\rr^d},\\
\phi_{\Hs}(.) &= (
\phi_\jn(.) \inds_{\Hs}(.))_{\nu  \in  \Se_j(\Hs)} \in
\R^{\rr^d}.
\end{align*}
and denote by $Y_{\Hs} = (Y_i \inds_{\Hs}(X_i))_{1 \leq i \leq
  n}$. 

\section{Construction of the local estimator $\eta^@$}\label{sec:estproc}

Assume we are under \Hp{CS1} and work with the Daubechies'
$r$-MRA of $\Lp_2(\R^d, \lambda)$. The
estimation procedure is local, so that we start by
selecting a point $x \in \A$. By construction, there exists $\Hs \in \info_j$ such that $x \in
\Hs$. We want to estimate $\eta$ at point $x$. As detailed in the Appendix, an
estimator of $\eta$ can be reduced to an estimator of the orthogonal projection
$\pr_j\eta$ of $\eta$ onto $\V_j$, modulo an error $\re_j\eta$, such
that $\abs{\re_j\eta} \leq M2^{-js}$ when $\eta$ belongs to the
generalized Lipschitz ball $\Lip^s(\Xs, M)$ of radius $M$. Now, we can write
\begin{align*}
\pr_j \eta(x) = \sum_{k \in \Zr^d} \alpha_\jk \phi_\jk(x) = \sum_{ k \in
  \Se_j(\Hs)} \alpha_\jk \phi_\jk(x) = \ms{\alpha_{\Hs}}{\phi_{\Hs}(x)}.
\end{align*}
This leaves us with exactly $\rr^d$ coefficients $\alpha_\jn, \nu \in
\Se_j(\Hs)$ to estimate, which are valid for any $x \in \Hs$. We evaluate these coefficients
by least-squares. Denote by $B_{\Hs} \in \Ma_{n ,\rr^d}$ the
matrix whose rows are the vectors $\phi_{\Hs}(X_i)^t$ for $1\leq i
\leq n$. Let us denote by $k_1, \ldots, k_{\rr^d}$ the elements of
$\Se_j(\Hs)$. Then we choose
\begin{align}\label{eq:lwdef}
\alpha_{\Hs}^{\diamond} &\in \arg \min_{a \in \R^{\rr^d}}
\sum_{i=1}^n \left(Y_i - \sum_{t=1}^{\rr^d} a_t \phi_{j,k_t}(X_i) \right)^2\inds_{\Hs}(X_i)\nonumber\\
&= \arg \min_{a \in \R^{\rr^d}} \normL{Y_{\Hs} - B_{\Hs}\cdot a }{\ell_2(\R^{n})}^2,
\end{align}
where we set $\alpha^{\diamond}_{\Hs}=0$ if the $\arg\min$ above contains more
than one element. Let us write $Q_{\Hs} = B_{\Hs}^t \cdot
B_{\Hs} / n \in \Ma_{\rr^d, \rr^d}$. As is well
known, when $Q_{\Hs}$ is invertible, the $\arg\min$ on the rhs of
\ref{eq:lwdef} admits one single element which writes as follows,
\begin{align}\label{eq:lwsol}
\alpha^\diamond_{\Hs} = Q_{\Hs}^{-1} \cdot \frac 1n B_{\Hs}^t\cdot Y_{\Hs}.
\end{align}
Naturally, we will denote the corresponding estimator of
$\pr_j\eta$ at point $x$ by $\eta_\Hs^\diamond(x) =
\ms{\alpha_\Hs^\diamond}{\phi_\Hs(x)}$.\\ 
We now introduce a thresholded version of $\eta_{\Hs}^{\diamond}$
based on the spectral thresholding of $Q_{\Hs}$. We denote by
$\lambda_{\min}(Q_\Hs)$ the smallest eigenvalue of $Q_\Hs$ in
the case where $r\geq 2$, when $Q_\Hs$ is actually a matrix, and
$Q_\Hs$ itself in the case where $r=1$, when it is a
real number. Furthermore, we define
\begin{align}\label{eq:defloc}
\eta_{\Hs}^{@}(x) = 
\begin{cases}
0 &\text{if } \pi_n^{-1} > \lambda_{\min}(Q_{\Hs})\\
\eta_\Hs^\diamond(x) &\text{otherwise}
\end{cases},
\end{align}
where $\pi_n$ is a tuning parameter. In practice, and unless otherwise
stated, we choose $\pi_n = \log n$. Moreover, we assume throughout the paper that $n$ is large
enough so that $\pi_n^{-1} \leq \min(\tfrac{g_{\min}}2,1)$, where, for
reasons that will clarified later, we have denoted,
\begin{align}\label{eq:gmin}
g_{\min} := \mu_{\min}c_{\min},
\end{align}
and $c_{\min}$ stands for the strictly positive constant defined in
the proof of \ref{prop:support}. Ultimately, the
estimator $\eta_j^{@}$ of $\pr_j\eta$ is defined as,
\begin{align}\label{eq:estimatordef}
\eta_j^@(x) &= \sum_{\Hs \in \info_j} \eta_{\Hs}^{@}(x)\inds_{\Hs}(x),
& x \in \Xs.
\end{align}

\section{The results}\label{sec:res} 

Let $r$ be a natural integer, denote by $\Ps$ the set
of all distributions on $\Xs \times \Y$ and write
\begin{align}\label{eq:probanotation}
\Ps(\mbf{CS1}, \mbf{H_s^r}) &:= \{ \Pb \in \Ps: \text{ \Hp{CS1} and \Hp{H_s^r} hold true} \}.
\end{align}
Furthermore, we define $j_r, j_s, J$ and $t(n)$ such that,
\begin{align*}
2^{j_{r}} &= \floor{n^{\frac 1{2r+d}}}, & 2^{j_s} &= \floor{n^{\frac
    1{2s+d}}},\\
2^{Jd} &= \floor{n t(n)^{-2}}, & t(n)^2 &= \kappa \pi_n^2\log n,
\end{align*}
where $\kappa$ is a positive real number to be chosen later. In addition,
we write $\J_n = \{j_{r}, j_{r}+1, \ldots, J-1, J\}$. Notice that $j_s$
strikes the balance between bias and variance in the sense that, for
$\log n \geq (2s+d) \log 2$ and $s \in (0,r)$, one has got 
\begin{subequations}
\begin{align}
n^{-\frac 12} 2^{j_s\frac d2} &\leq 2^{-j_ss},\label{eq:sa}\\
2^{-j_ss}&\leq 2^{r+\frac d2} 2^{j_s \frac
  d2}n^{-\frac 12},\label{eq:sb}\\
2^{-j_ss} &\leq 2^r n^{-\frac s{2s+d}}.\label{eq:sc}
\end{align}
\end{subequations}
Throughout the
sequel, we assume that $n$ is large enough so that the latter
inequalities hold true. Our first result gives an upper bound on the probability of deviation
of $\eta_j^@$ form $\eta$ at a point $x \in \A$.
\begin{theorem}\label{th:main}
Fix $r \in \N$ and assume we are under \Hp{CS1} and \Hp{H_s^r}. Recall
that $\eta_j^{@}$ is defined in \ref{eq:estimatordef}. Then, for all
$j\in \J_n$, all $ \delta >
2M 2^{-js} \linebreak[2]\max(1,3\pi_n \rr^{d}
    \mu_{\max})$ and all
$x \in \A$, we have got
\begin{align}
&\sup_{\Pb \in \Ps(\mbf{CS1}, \mbf{H_s^r})}\Pb^{\otimes n}( \abs{
  \eta(x) - \eta_j^@(x) } \geq \delta)\nonumber\\
&\leq 2\rr^{2d} \exp\left( - n2^{-jd} \frac{\pi_n^{-2}}{
    2\mu_{\max}\rr^{4d} + \frac 43 \rr^{2d} \pi_n^{-1} } \right
)\ind{\delta \leq M}
 + \rr^d \Lambda\left(  \frac{\delta 2^{-j\frac d2}}{2 \pi_n \rr^{d}}  \right),\label{eq:expconcentrationproba}
\end{align}  
where $\Lambda$ is defined as follows,
\begin{align*}
\Lambda(\delta ) = 
\begin{dcases}
2\exp \left( - \frac{n
    \delta^2 }{18K^2 \mu_{\max} + 4 K 2^{j\frac d2}\delta}
\right), &\text{under \Hp{N1}}\\
\\
1 \wedge \left\{ \frac{2 \sigma (\mu_{\max} + 2^{j\frac d2}
    \delta)^{\frac 12}}{\delta \sqrt{2\pi n}} \exp\left(-
    \frac{n \delta^2 \sigma^{-2}}{\mu_{\max} + 2^{j\frac d2} \delta}  \right)
\right \}\\
\qquad + 2 \exp\left( -\frac{n \delta^2}{2\mu_{\max}  +
    \frac 43 2^{j\frac d2}\delta }  \right),  &\text{under \Hp{N2}}
\end{dcases}
\end{align*}
\end{theorem}
As a consequence of the above theorem, we can deduce the (near)
minimax optimality of $\eta_{j_s}^@$ over generalized
Lipschitz balls.
\begin{corollary}\label{co:main}
Fix $r \in \N$ and assume we are under \Hp{CS1} and \Hp{H_s^r}. Then,
for any $p \in [1,\infty)$ and $j \in \J_n$, one has got
\begin{align}\label{eq:ubj}
\sup_{\Pb \in \Ps(\mbf{CS1}, \mbf{H_s^r})}\Exp^{\otimes n}\normL{ \eta - \eta_j^@ }{ \Lp_p( \Xs, \mu ) }^p &\leq C(p)
\pi_n^p\max\left( 2^{-js}, \frac{2^{j\frac d2}}{\sqrt{n}}\right)^p,
\end{align}
where $\eta_j^@$ and $C(p)$ are defined in
\ref{eq:estimatordef} and \ref{prop:espx} below,
respectively. \textsl{A fortiori}, when $s$ is \textbf{known}, we can
choose $j = j_s$ and apply \ref{eq:sa} and \ref{eq:sc} above to obtain
\begin{align*}
\sup_{\Pb \in \Ps(\mbf{CS1}, \mbf{H_s^r})}\Exp^{\otimes n}\normL{ \eta - \eta_{j_s}^@ }{
  \Lp_p( \Xs, \mu ) }^p &\leq C(p) 2^{rp} \pi_n^p n^{- \frac{sp}{2s+d}}.
\end{align*}
This, together with the lower-bound of \ref{th:lowerbound}, proves
that $\eta_{j_s}^@$ is (nearly) minimax optimal over the generalized
Lipschitz ball $\Lip^s(\Xs,M)$ of radius $M$.
\end{corollary}

The next Theorem shows that the approximation level $j$ can be
determined from the data $\D_n$ so that we obtain adaptation over a wide
generalized Lipschitz scale. 

\begin{theorem}\label{th:adaptation}
Fix $r \in \N$ and assume we are under \Hp{CS1} and \Hp{H_s^r}. We define
\begin{align*}
g(j,k) &:= 
\left( \frac{2^{j \frac
    d2}}{\sqrt{n}} t(n) + \frac{2^{k \frac
    d2}}{\sqrt{n}} t(n) \right),\\
j^@(x) &:= \inf\{ j \in \J_n: \abs{\eta^@_j(x) - \eta^@_k(x)} \leq
g(j,k), \forall k \in \J_n, k > j   \}, \quad x \in \A,
\end{align*}
where $\eta_j^@$ is defined in \ref{eq:estimatordef} and $\inf
\emptyset = \max(\J_n) = J$. If $\kappa$ is chosen large
enough, meaning $\kappa \geq \tfrac p2 C_9^{-1}$, where $C_9$ is defined in \ref{prop:expbound}, then we obtain
\begin{align*}
\sup_{\Pb \in \Ps(\mbf{CS1}, \mbf{H_s^r})}\Exp^{\otimes n} \normL{
  \eta_{j^@(.)}^@(.) - \eta(.) }{\Lp_p(\Xs, \mu)}^p
\leq 5^p 2^{rp} t(n)^pn^{-\frac{sp}{2s+d}}.
\end{align*}
So that $\eta_{j^@(.)}^@(.)$ is a nearly minimax adaptive estimator of $\eta$ over the
generalized Lipschitz scale ${\displaystyle \bigcup_{0<s<r} \Lip^s(\Xs,M)}$.
\end{theorem}

Finally, we prove that $\eta^@$ is indeed (nearly) minimax optimal by giving
the corresponding lower-bound result.

\begin{theorem}\label{th:lowerbound}
Assume we are under \Hp{CS1} and \Hp{H_s^r}. We write $\inf_{\theta_n}$ the infinimum over all estimators $\theta_n$
of $\eta$, that is all measurable functions of the data $\D_n$. Then,
for $d\geq 1$, $s> 0$, we
have, for all $1 \leq p < \infty$,
\begin{align*}
\inf_{\theta_n}\sup_{\Pb \in \Ps(\mbf{CS1}, \mbf{H_s^r})}\Exp^{\otimes
n} \normL{\theta_n-\eta }{\Lp_p(\Xs,\mu)}^p \gtrsim n^{-\frac{sp}{2s+d}}. 
\end{align*}
\end{theorem}

The next section shows how these results can be improved in the case
where we benefit from additional information on $\mu$ or $\eta$.

\section{Refinement of the results}\label{sec:refinment}

As can be seen from \ref{co:main} and \ref{th:adaptation} above,
$\pi_n$ appears as a multiplicative factor in the upper-bounds and
thus deteriorates them by a multiplicative $\log n$ term. However,
this needs not be the case, and under appropriate additional
assumptions, $\pi_n$ can be chosen to be a constant. Consider indeed the
following two assumptions.
\begin{compactdesc}
\item[\Hp{O1}] We know $\mu_{\min}^* \in \R$, such that $0 <
  \mu_{\min}^* \leq \mu_{\min}$.
\item[\Hp{O2}] We know a finite positive real number $M$ such that
  $\normL{\eta}{\Lp_{\infty}(\Xs, \lambda)} \leq M$.
\end{compactdesc}
Under \Hp{O1}, we know a lower bound $\mu^*_{\min}$ of $\mu_{\min}$, and therefore a lower bound
$g_{\min}^*$ of $g_{\min}$ (see \ref{eq:gmin}). Under \Hp{O1}, we will thus choose $\pi_n^{-1} =
\min(\tfrac{g^*_{\min}}2, 1)$. It is straightforward to show that
\ref{th:main} is still valid with this new value of $\pi_n$ (see \ref{rem:updateO2} in the proof
of \ref{th:main}), and thus all the
subsequent results follow as well. Under \Hp{O2}, we know an upper bound $M$ of the essential supremum of $\eta$ on
$\Xs$. In that case, we redefine
\begin{align}\label{eq:O2mod}
\eta_{\Hs}^{@}(x) = T_M( \eta_\Hs^{\diamond}(x))
\ind{\lambda_{\min}(Q_{\Hs}) > 0},
\end{align}
where, for any $z \in \R$, we have written $T_M(z) = z \ind{ \abs{z} \leq
  M} + M \textrm{sign}(z) \ind{ \abs{z} > M}$. Once again, it is straightforward to show that
\ref{th:main} is now valid with $\pi_n^{-1} =
\min(\tfrac{g_{\min}}2, 1)$ and $2M$ in place of $M$ in the indicator
function on the rhs of \ref{eq:expconcentrationproba} (see
\ref{rem:updateO2} in the proof of \ref{th:main}), and thus all the
subsequent results follow as well.\\
Notice that $\pi_n$ is an absolute constant under \Hp{O1} and \Hp{O2},
while it is an increasing sequence of $n$ to be fine-tuned by the
statistician otherwise. Hence $\pi_n$ appears to be the price to pay for not
knowing a lower bound of $\mu_{\min}$ or an upper bound of the
essential supremum of $\eta$ on $\Xs$.

\section{Relaxation of assumption \Hp{S1}}\label{sec:relaxS1}

\subsection{The problem}
Now, we would like to relax assumption \Hp{S1} and allow for $\A$ to
be an unknown subset of $\Xs$, eventually disconnected. Under
\Hp{CS1}, the success of $\eta^@$ stems from the fact that it is
constructed upon an approximation grid of the form $2^{-j}\Zr^d\cap
[0,1]^d$, whose edges coincide exactly with the boundary of $\A$. In the case where $\A$ is
unknown, some cells of the lattice might straddle the boundary of $\A$
and thus require a new treatment.\\
In order to handle this new configuration, we will need to make a
smoothness assumption on the boundary of $\A$ and allow for the
estimation cells to move with the point at which we
want to estimate $\eta$. Ultimately, we devise a new estimator
$\eta^{\maltese}$ of $\eta$ which is built upon a moving approximation
grid. In fact, this new estimation method ensures that the point $x$ at which we want
to estimate $\eta$ always belongs to a cell $\Hs$ of $\info_j$ at
resolution level $j$, whose center belongs to $\A$. This will
ensure that local regressions performed on cells that straddle the
boundary of $\A$ are still meaningful.\\
The smoothness assumption we will make on $\A$ might be
compared to the support assumption made in
\mycite[eq.~(2.1)]{Audibert2007} in the classification context. In
substance, it is assumed in \mycite{Audibert2007} that $\A$ is locally
ball-shaped to be compatible with the ball-shaped support of
the LPE kernel, which they use to estimate $\eta$. In our
case, we perform estimation with multi-dimensional scaling functions
whose supports are cube-shaped and will thus assume
that $\A$ is locally cube-shaped.  

\subsection{Smoothness assumption on $\A$}

Let us now make these informal
arguments more precise. To that end we introduce assumption \Hp{S2} as
an alternative to \Hp{S1} above. Fix an absolute constant $\m_0 \in
(0,1)$ and recall that $2^{j_{s}} = \floor{n^{\frac 1{2s+d}}}$. With these notations,
\Hp{S2} goes as follows,
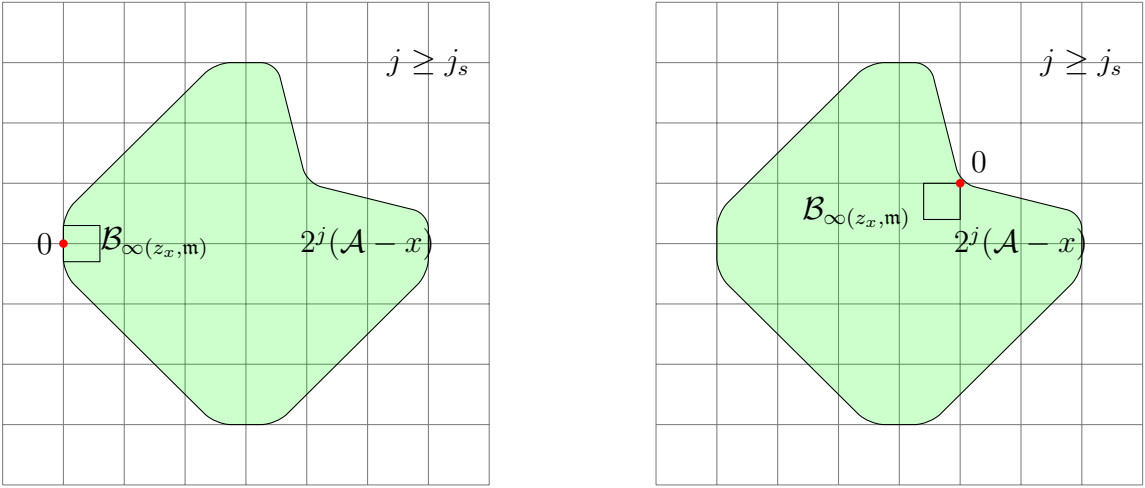
\begin{figure}[htb]
\begin{center}
\begin{minipage}[c]{\textwidth}
\begin{minipage}[c]{0.5\textwidth}
\begin{tikzpicture}[scale=.8]
\draw[thin,color=gray] (-4,-4) grid (4,4);

\draw[fill = green, fill opacity=0.2, rounded corners = 0.2cm] (3, 0.5)
-- (1, 1)-- (0.5, 3) -- (-0.5,3) -- (-3,0.5) -- (-3, -0.5) -- (-0.5, -3) --
(0.5, -3) -- (3, -0.5) -- cycle; 

\draw (-3, -0.3) rectangle (-2.4, 0.3);

\node at (2,0) {$2^{j}(\A-x)$};
\node at (-1.5,0) {$\Bo_{\infty(z_x, \m)}$};
\node at (3,3) {$j\geq j_s$};

\fill [red] (-3, 0) circle (2pt);
\node at (-3,0) [left] {$0$};
\end{tikzpicture}
\end{minipage}
\hfill
\begin{minipage}[c]{0.5\textwidth}
\begin{tikzpicture}[scale=.8]
\draw[thin,color=gray] (-4,-4) grid (4,4);

\draw[fill = green, fill opacity=0.2, rounded corners = 0.2cm] (3, 0.5)
-- (1, 1)-- (0.5, 3) -- (-0.5,3) -- (-3,0.5) -- (-3, -0.5) -- (-0.5, -3) --
(0.5, -3) -- (3, -0.5) -- cycle; 

\draw (0.4, 0.4) rectangle (1, 1);

\node at (2,0) {$2^{j}(\A-x)$};
\node at (-0.7,0.5) {$\Bo_{\infty(z_x, \m)}$};
\node at (3,3) {$j\geq j_s$};

\fill [red] (1, 1) circle (2pt);
\node at (1,1) [above right] {$0$};

\end{tikzpicture}
\end{minipage}
\end{minipage}
\end{center}
\caption{\Hp{S2} allows for $\A$ to be non-convex and eventually disconnected.}\label{figure:Adescription}
\end{figure}

\begin{compactdesc}
\item[\Hp{S2}] $\Xs = \R^d$ and $\A$
  belongs to $\As_{j_{s}}$, where 
\begin{align*}
\begin{split}
\As_{j_{s}}:=\{ \A \subset \R^d: &\exists \m \geq \m_0, \forall x
\in \A, \\
&\exists z_x \in \R^d, 0 \in \Bo_{\infty}(z_x,\m) \subset 2^{j_{s}} (\A -x)  \},
\end{split}
\end{align*}
\end{compactdesc}
In words, \Hp{S2} means that if we
zoom close enough to any $x \in \A$, we can find a hypercube
$\Bo_\infty(z_x, \m)$ that contains $x$ and is a subset of
$\A$. Notice readily that for
all $j_1 \geq j_2$, the component of $2^{j_2}(\A -x)$ that contains
$0$ is a subset of the component of $2^{j_1}(\A - x)$ that contains
$0$, so that $\As_{j_2} \subset \As_{j_1}$. Therefore $\As_{j_s}$ grows
with $n$ and shrinks with $s$. Of course, \Hp{S1} is a particular case of \Hp{S2}. Setting \Hp{S2}
allows $\A$ to be unknown and belong to a wide class of subsets of
$\R^d$, eventually disconnected (see \ref{figure:Adescription}).\\
In the sequel, we will conveniently refer by \Hp{CS2} to the set of
assumptions \Hp{D1}, \Hp{S2}, \Hp{N1} or \Hp{N2}.

\subsection{Moving local estimation under \Hp{CS2}}
As detailed above, $\eta^{\maltese}$ is obtained by local regression
on a moving approximation grid. Let us describe the construction of
$\eta^{\maltese}$ more precisely.\\
First of all, we split the sample into two pieces. For simplicity, let
us assume that we dispose of $2n$ data points. The first half of the
sample points, which we denote by $\D_n' = \{ (X_i', Y_i'), i=1,\ldots,n
\}$, will be used to identify the support $\A$ of $\mu$, while
the second half, which we denote by $\D_n = \{ (X_i,Y_i),
i=1,\ldots,n\}$, will be used to estimate the scaling functions
coefficients by local regressions.\\  
Let us denote by $\Hs_0$ the cell $2^{-j}[0,1]^d$ of the lattice
$2^{-j}\Zr^d$ at resolution $j$. And denote by $\Hs_0(x)$ the same
cell centered in $x$, that is $\Hs_0(x) = x - 2^{-j-1} +
2^{-j}[0,1]^d$. Then, the construction of $\eta^{\maltese}_j(x)$ at a
point $x \in \R^d$ goes as follows. \begin{inparaenum}[(i)]\item If
  none of the design points $(X_i')$ of the sample $\D_n'$ lie in $\Hs_0(x)$, then take
  $\eta^{\maltese}(x)=0$. \item If one or more design points of the
  sample $\D_n'$ lie in
  $\Hs_0(x)$, we select one of them and denote it by $X_{i_x}'$ (the
  selection procedure is of no importance beyond computational considerations). By
  construction, $x$ belongs to the cell $\Hs_0(X_{i_x}')$ centered in
  $X_{i_x}' \in \A$. Since $X_{i_x}'$ belongs to $\A$, it makes sense
  to perform a local regression on $\Hs_0(X_{i_x}')$ with the sample
  points $\D_n$, which gives rise
  to an estimator $\eta^{\maltese}$ of $\eta$ valid at any point of
  $\Hs_0(X_{i_x}')\cap \A$. \end{inparaenum} It is noteworthy that
this procedure uses the sample $\D_n'$ to identify the support $\A$ of
$\mu$.\\ 
Interestingly, the above estimation procedure requires at most as many
regressions as there are data points in $\D_n'$ to return an estimator
$\eta^{\maltese}$ of $\eta$ at every single point $x \in \A$. It is
therefore computationally more efficient
than any other kernel estimator, such as the LPE. The
computational performance of $\eta^{\maltese}$ can in fact be further improved
in the sense that the local regression on the cell $\Hs_0(X_i')$ can be
omitted if the cell $\Hs_0(X_i')$ is itself included in the union of
cells centered at other design points of $\D_n'$. In particular, we
can choose $X_{i_x}'$ to be a design point $X_i'$ of $\D_n'$ that
belongs to $\Hs_0(x)$ and for which a local regression has already
been performed, if it exists, or any one of the $X_i'$ that belong to
$\Hs_0(x)$ otherwise.\\
Intuitively, the computational efficiency of $\eta^{\maltese}$ stems
from the fact that the design points $(X_i')$ provide some
valuable information on the unknown support
$\A$ of $\mu$, which can be exploited under \Hp{CS2}. In particular, and as we
will see below, \Hp{D1} guarantees that the design points of $\D_n'$ populate
$\A$ densely enough so that, as long as $j\leq J$, the
cells $\Hs_0(X_i')$, $1 \leq i \leq n$, form a
cover of $\A$, modulo a set whose $\mu$-measure decreases almost
exponentially fast toward zero with $n$.

\subsection{Construction of the local estimator $\eta^{\maltese}$}

Assume we are under \Hp{S2} and work with the Daubechies' $r$-MRA of
$\Lp_2(\R^d,\lambda)$. Obviously, shifting the approximation grid is equivalent to
shifting the data points $(X_i)$ of $\D_n$ and keeping the lattice fixed. For
ease of notations and clarity, we adopt this second point of view. In
order to compute $\eta^{\maltese}$ at a point $x \in \Hs_0(X_{i_x}')
\cap \A$, we want to shift the design points in such a way that
$X_{i_x}'$ falls right in the middle of $\Hs_0$. In
other words, we want $X_{i_x}'$ to be shifted at point $2^{-j-1}
\in \R^d$ (whose coordinates are worth
$2^{-j-1}\in\R$). This corresponds to the change of variable $\tilde X_i = X_i - (X_{i_x}' -
2^{-j-1})$, where we have denoted by $X_i$ and $\tilde X_i$ the
representations of a same data point in the canonical and shifted
coordinate systems of $\R^d$, respectively. In order to compute
$\eta^{\maltese}$ at point $x \in \Hs_0(X_{i_x}')\cap \A$, it is therefore
enough to perform a local regression on $\Hs_0$ against the shifted
data points,
\begin{align*}
\tilde \D_x = \{ ( \tilde X_i, Y_i), i = 1,\ldots,n\}.
\end{align*}
For the sake of concision, we will denote by $\tilde u = u - (X_{i_x}' - 2^{-j-1})$ the coordinate
representation of a point $u$ in the shifted coordinate system of
$\R^d$. Let us denote by $k_1, \ldots, k_{R^d}$ the elements of
$\Se_j(\Hs_0)$. With these notations, \ref{eq:lwdef} must be corrected and written
as 
\begin{align}\label{eq:locregmaltese}
\alpha^{\diamond}_{\Hs_0} \in \arg \min_{a \in \R^{R^d}}
\sum_{i=1}^n\left( Y_i - \sum_{t=1}^{R^d} a_t
  \phi_{j,k_t}(\tilde X_i)\right)^2 \inds_{\Hs}(\tilde X_i),
\end{align}   
where we set $\alpha^{\diamond}_{\Hs_0}=0$ if the $\arg\min$ above contains more
than one element. The notations introduced in \ref{generalnotations} can be
updated to this new setting as follows. $B_{\Hs_0}$ stands now for the random
matrix of $\Ma_{n, R^d}$ whose rows are the $\phi_{\Hs_0}(\tilde
X_i)^t$, $i=1,\ldots,n$. In addition, we
recall that we have defined $Q_{\Hs_0} = B_{\Hs_0}^t \cdot B_{\Hs_0}/n \in \Ma_{R^d,R^d}$.
Its coefficients write thus as
\begin{align*}
[Q_{\Hs_0}]_{\nu,\nu'} &= \frac{1}{n} \sum_{i=1}^n \phi_{j,\nu}(\tilde X_i) \phi_{j,\nu'}(\tilde X_i)
\inds_{\Hs_0}(\tilde X_i),
&\nu,\nu' \in \Se_j(\Hs_0).
\end{align*} 
Notice here that $\Se_j(\Hs_0) = \{ \nu \in \Zr^d: 2^{-1} \in
\Supp \phi_{\nu}\}$, which neither depends on $j$ nor $x$. Therefore, and for later reference, we denote 
\begin{align}\label{eq:fdeldef}
\fdel := \{ \nu \in \Zr^d: 2^{-1} \in \Supp \phi_{\nu}\}, 
\end{align}
In addition, if we write $Y_{\Hs_0} = (Y_i \inds_{\Hs_0}(\tilde
X_i))_{1\leq i \leq n}$, then \ref{eq:lwsol} still holds true when the
solution to \ref{eq:locregmaltese} is unique. So
that, for all $x \in \Hs_0(X_{i_x})\cap \A$, we can write
$\eta_{\Hs_0}^{\diamond}(\tilde x) =
\ms{\alpha_{\Hs_0}^{\diamond}}{\phi_{\Hs_0}(\tilde x)}$. Finally \ref{eq:defloc} remains
valid with $X_i$ replaced by $\tilde
X_i$ and $\Hs$ by $\Hs_0$, $\eta_{\Hs_0}^@$ redefined as
$\eta_{\Hs_0}^{\maltese}$ and $g_{\min}$ redefined as
\begin{align}\label{eq:gminbis}
g_{\min} = \mu_{\min}c_{\min},
\end{align}
where $c_{\min}$ is the strictly positive constant defined in
\ref{lem:lower} below. So that ultimately, the estimator $\eta_j^{\maltese}$ of
$\pr_j \eta$ at a point $x\in\R^d$ writes as
\begin{align}\label{eq:estimatordefbis}
\eta_j^{\maltese}(x) &= \eta_{\Hs_0}^{\maltese}(\tilde x), & x \in \Xs.
\end{align}
Notice that by contrast with \ref{eq:estimatordef} above, the sum over
the hypercubes of $\info_j$ has disappeared. This is due to the fact
that the approximation grid moves with $x$ so that we end up
virtually always performing estimation on the same hypercube
$\Hs_0$. 

\subsection{The results}
Interestingly, $\eta^{\maltese}$ still verifies similar results as the
ones described in \ref{sec:res}. To be more precise, recall that we
work with a sample of size $2n$ broken up into two pieces $\D_n$ and
$\D_n'$ of size $n$. Let us redefine $\J_n$ so
that $\J_n = \{j_{s}, j_{s}+1, \ldots, J-1, J\}$ where $2^{j_{s}} = \floor{n^{\frac
    1{2s+d}}}$. Then, we obtain the
following result in place of \ref{th:main}. 
\begin{theorem}\label{th:mainbis}
Fix $r \in \N$ and assume we are under \Hp{CS2} and \Hp{H_s^r}. Recall
that $\eta_j^{\maltese}$ is defined in \ref{eq:estimatordefbis}. Then, for all
$j\in \J_n$, all $ \delta >
2M 2^{-js} \linebreak[2]\max(1,3\pi_n \rr^{d}
    \mu_{\max})$ and all
$x \in \A$, we have got
\begin{align*}
&\sup_{\Pb \in \Ps(\mbf{CS2}, \mbf{H_s^r})}\Pb^{\otimes n}( \abs{
  \eta(x) - \eta_j^{\maltese}(x) } \geq \delta)\\
&\leq 3\rr^{2d} \exp\left( - n2^{-jd} \frac{\pi_n^{-2}}{
    2\mu_{\max}\rr^{4d} + \frac 43 \rr^{2d} \pi_n^{-1} } \right
)\ind{\delta \leq M} \\
 &+ \rr^d \Lambda \left(  \frac{\delta 2^{-j\frac d2}}{2 \pi_n \rr^{d}}  \right),
\end{align*}  
where $\Lambda$ has been defined in \ref{th:main}. 
\end{theorem} 
Left aside the fact that $\eta^{\maltese}$ is constructed upon a
sample of size $2n$, the sole difference with the result of
\ref{th:main} is that the leading constant in front of the
exponential on the second line has changed from $2R^d$ to
$3R^d$. Furthermore, it is straightforward to deduce from \ref{th:mainbis} results
similar to \ref{co:main}, \ref{th:adaptation} and \ref{th:lowerbound}, and a
fortiori the refined results obtained in
\ref{sec:refinment}, for $\eta^{\maltese}$ under \Hp{CS2}. The proofs
of these results for $\eta^{\maltese}$ under the set of assumptions
\Hp{CS2} follow,
for the most part, exactly the same lines as the proofs given for $\eta^{@}$ under
\Hp{CS1}. Details can be found in \ref{sec:proofCS2}. 

\section{Classification via local multi-resolution projections}\label{sec:classifres}
Recall from \mycite{Audibert2007} that the margin assumption can be written as,
\begin{compactdesc}
\item[\Hp{MA}] There exist constants $C_* > 0$ and $\vartheta \geq 0$ such that
\begin{align*}
\Pb( 0 < \abs{ 2\eta(X) - 1 } \leq t) \leq C_* t^{\vartheta}, \quad \forall t>0.
\end{align*}
\end{compactdesc}
The binary classification setting corresponds to \Hp{CS2}, under
assumptions \Hp{N1} and \Hp{O2}. Notice besides that we have
$K=1$ in \Hp{N1} and $M=1$ in
\Hp{H_s^r}. Since we are under \Hp{O2}, it
follows from \ref{sec:refinment} that $\pi_n = \pi_0 = \min(1,
\tfrac{g_{\min}}2)$ is independent of $n$ and $\eta^{\maltese}$ is
capped at $M=1$ as in \ref{eq:O2mod}. For the sake
of coherence, we denote by
$j^{\maltese}$ the adaptive resolution level built upon
$\eta^{\maltese}$, as described in \ref{th:adaptation}, and define
$\Ps(\mbf{CS2}, \mbf{H_s^r})$ by analogy with \ref{eq:probanotation}
above. Finally, we recall that $\eta^{\maltese}$ is built upon a
sample of size $2n$ split into two sub-samples $\D_n$ and $\D_n'$ of
size $n$.\\ 
As a consequence of \ref{th:mainbis}, we can use the 
plug-in classifier built upon $\eta^{\maltese}$ to obtain similar results as the ones given in
\mycite[Lemma~3.1]{Audibert2007} for LPE based plug-in
classifiers.  
\begin{corollary}\label{co:mainclassif}
Fix $r \in \N$ and assume we are in the binary classification
setting. Assume moreover that \Hp{H_s^r} and \Hp{MA} hold
true. Consider the plug-in classifiers $h^{\maltese}_{j_s}(.) = \ind{ \eta^{\maltese}_{j_s}(.) \geq
  \frac 12}$ and $h^{\maltese}_{j^{\maltese}}(.) = \ind{ \eta^{\maltese}_{j^{\maltese}(.)}(.) \geq
  \frac 12}$ . Then, as soon as $\kappa > C_0(1+\vartheta)$, we have
\begin{align}
\sup_{\Pb \in \Ps(\mbf{CS2}, \mbf{H_s^r}, \mbf{MA})} \ct(h^{\maltese}_{j_s})&\leq C_1  n^{-\frac
  s{2s+d} (1+\vartheta)}, \label{eq:classifloss}\\ 
\sup_{\Pb \in \Ps(\mbf{CS2}, \mbf{H_s^r}, \mbf{MA})}
\ct(h^{\maltese}_{j^{\maltese}})&\leq C_{2} (\log n)^{\frac{1+\vartheta}2}  n^{-\frac s{2s+d} (1+\vartheta)}, \label{eq:classiflossadapt}
\end{align}
where the classification risk $\ct(.)$ has been defined in
\ref{sec:intro} and the constants $C_0,C_1,C_2$ are made explicit in
\mycite{Monnier2011} and only depend on $\mu_{\max}, \mu_{min}, r,d$ and $\vartheta$.
\end{corollary}
In fact, it can be shown that the classifiers $h^{\maltese}$
defined in \ref{co:mainclassif} are (nearly) minimax optimal. Proofs
of \ref{co:mainclassif} and the associated lower-bound can be found in
\mycite{Monnier2011}. 

\section{Simulation study}\label{sec:simuls}
In order to illustrate the performance of $\eta^@_{j^@}$, we have
carried out a simulation study in the regression setting in the
one-dimensional case, that is with $d=1$. As detailed earlier, the
sole purpose of this simulation is to show that \begin{inparaenum}[(1)] \item $\eta^@$ can
be easily implemented and is computationally efficient, \item $\eta^@$ works
well in practice in the case where the density of the design $\mu$ is
discontinuous, \item and to give an intuitive visual feel for
$\eta^@$, which is built upon the juxtaposition of local
regressions against a set of scaling functions. In particular, we run
our simulation against benchmark
signals, which allows to compare them
with the ones detailed in the literature for alternative kernel
estimators (see simulation study in \mycite{Lepski1997}, for example). \end{inparaenum} We have run
them under \Hp{CS1}, which corresponds to the case where $\eta^@_j$
can be completely computed with exactly $2^{j}$ regressions. We have
in particular $\Xs = [0,1] = \A$. We
focus on the functions $\eta$ introduced in \mycite{Donoho1994} and
used as a benchmark in numerous subsequent simulation studies. They are made
available through the Wavelab850 library freely available at
\url{http://www-stat.stanford.edu/~wavelab/}. In addition we
have chosen the noise $\xi$ to be standard normal, that is we are working under
\Hp{N2} with $\sigma=1$. In all cases, we have chosen the signal-to-noise ratio (SNR)
to be equal to $7$. To be more specific, we are working on a dyadic
grid $G$ of $[0,1]$ of resolution $2^{-15}$. We compute the
root-mean-squared-error (RMSE) of both the signal and the noise on that grid and
rescale the signal so that its RMSE be seven times
bigger than the one of the noise.
\begin{figure}[!p]
\begin{minipage}[c]{\textwidth}
\begin{minipage}[c]{\textwidth}
\fbox{\includegraphics[width=\textwidth, height = 0.5\textheight]{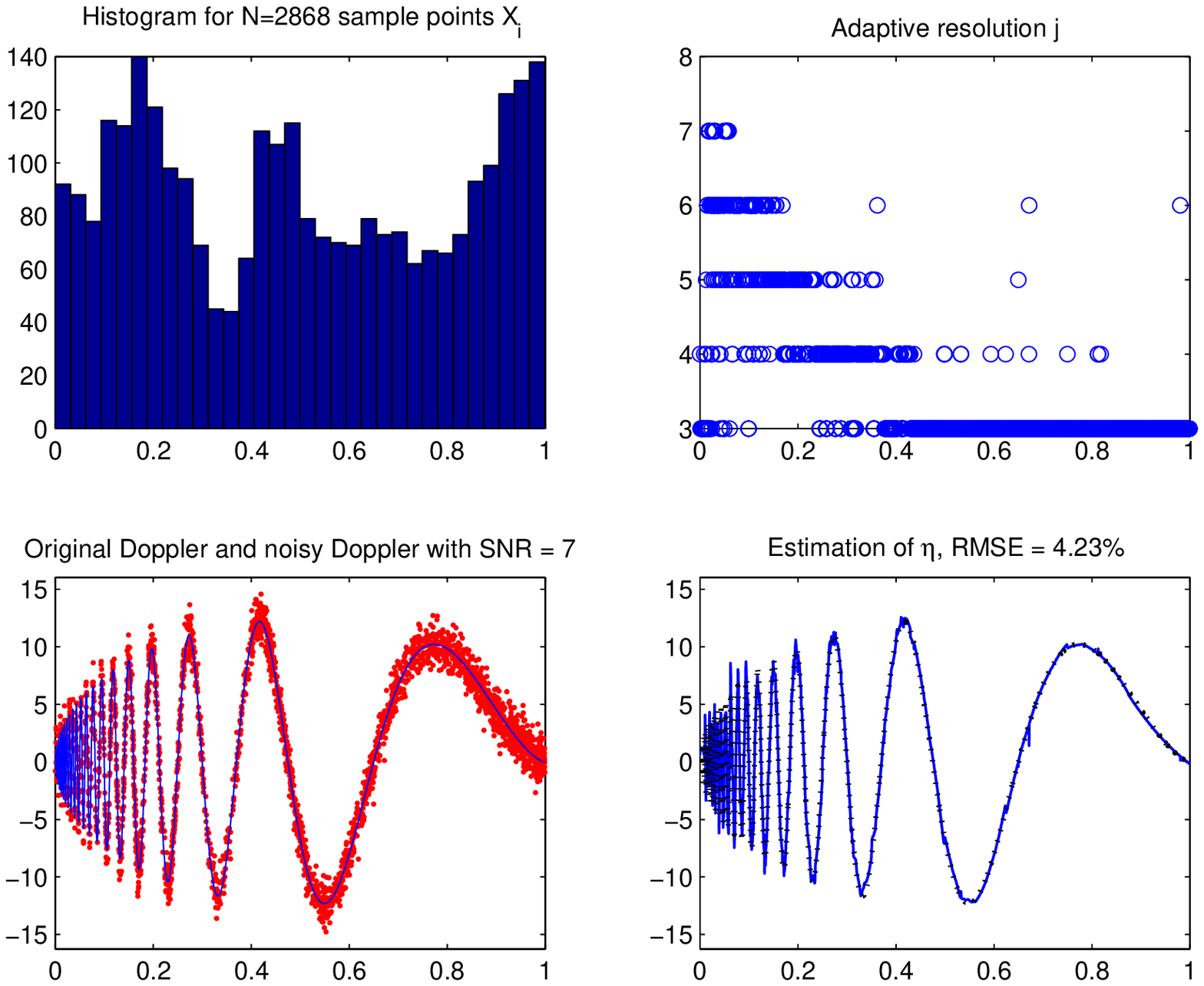}}
\end{minipage}
\begin{minipage}[c]{\textwidth}
\fbox{\includegraphics[width= \textwidth, height = 0.5\textheight]{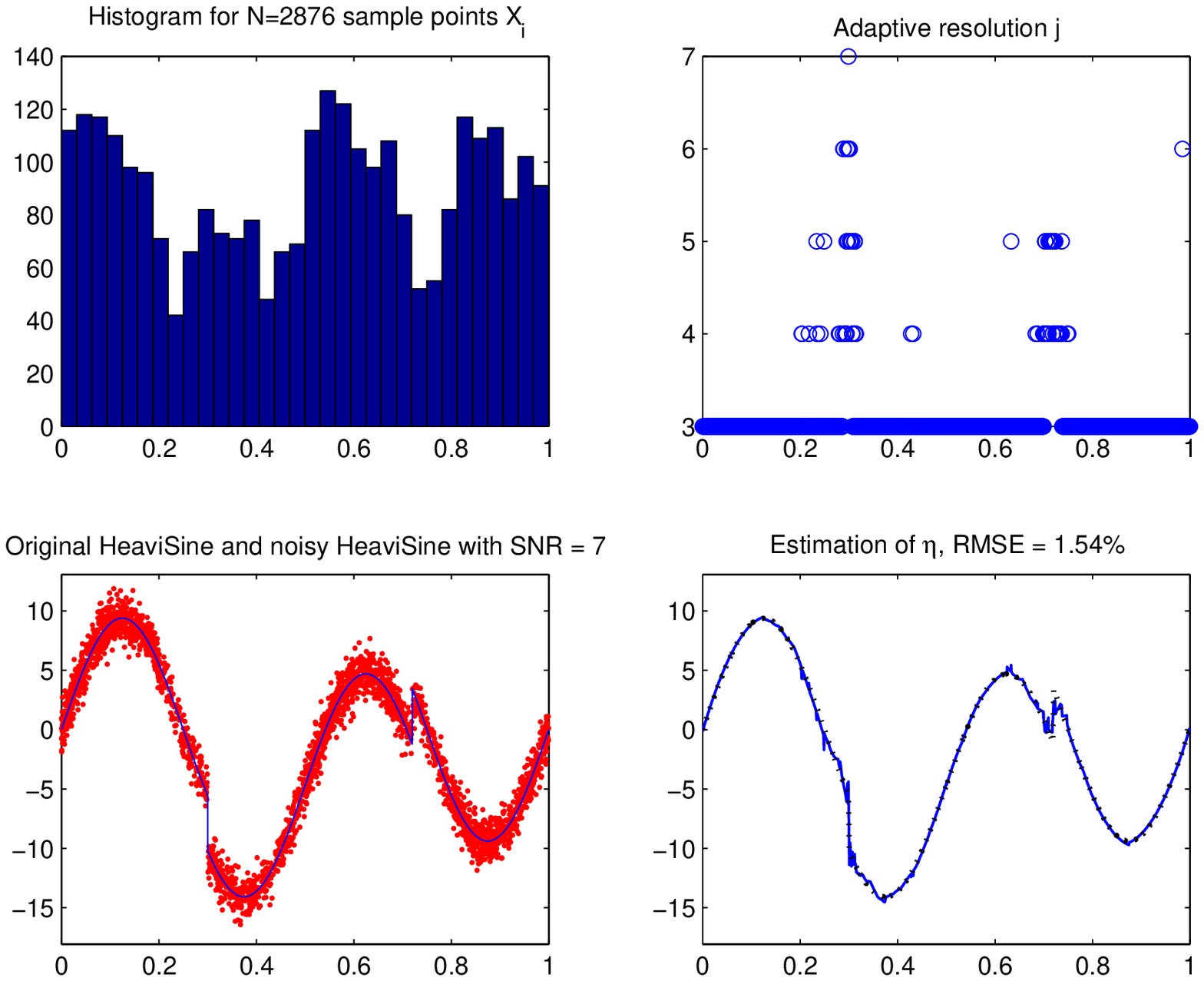}}
\end{minipage}
\end{minipage}
\end{figure}

\begin{figure}[!p]
\begin{minipage}[c]{\textwidth}
\begin{minipage}[c]{\textwidth}
\fbox{\includegraphics[width=\textwidth, height = 0.5\textheight]{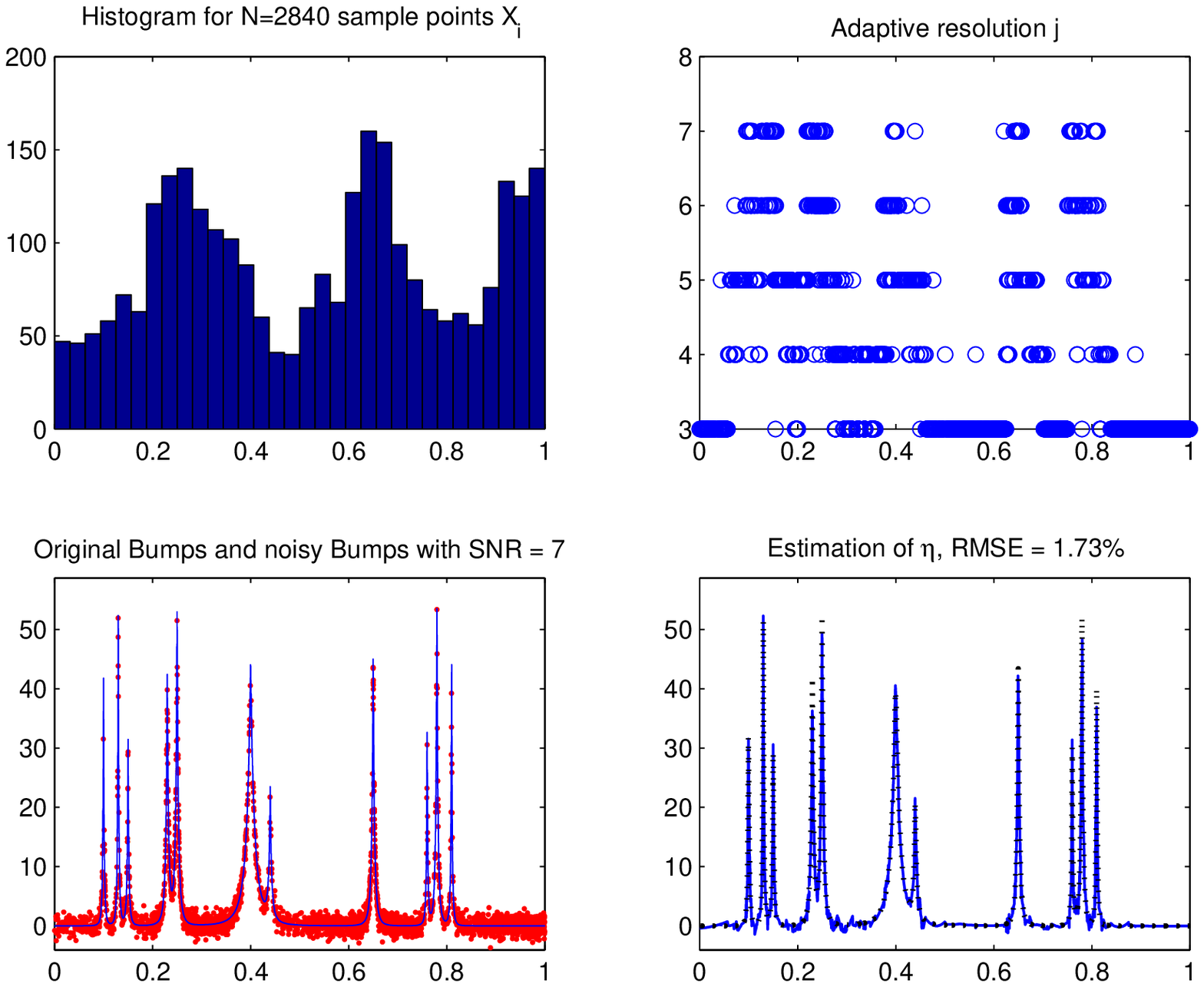}}
\end{minipage}
\begin{minipage}[c]{\textwidth}
\fbox{\includegraphics[width= \textwidth, height = 0.5\textheight]{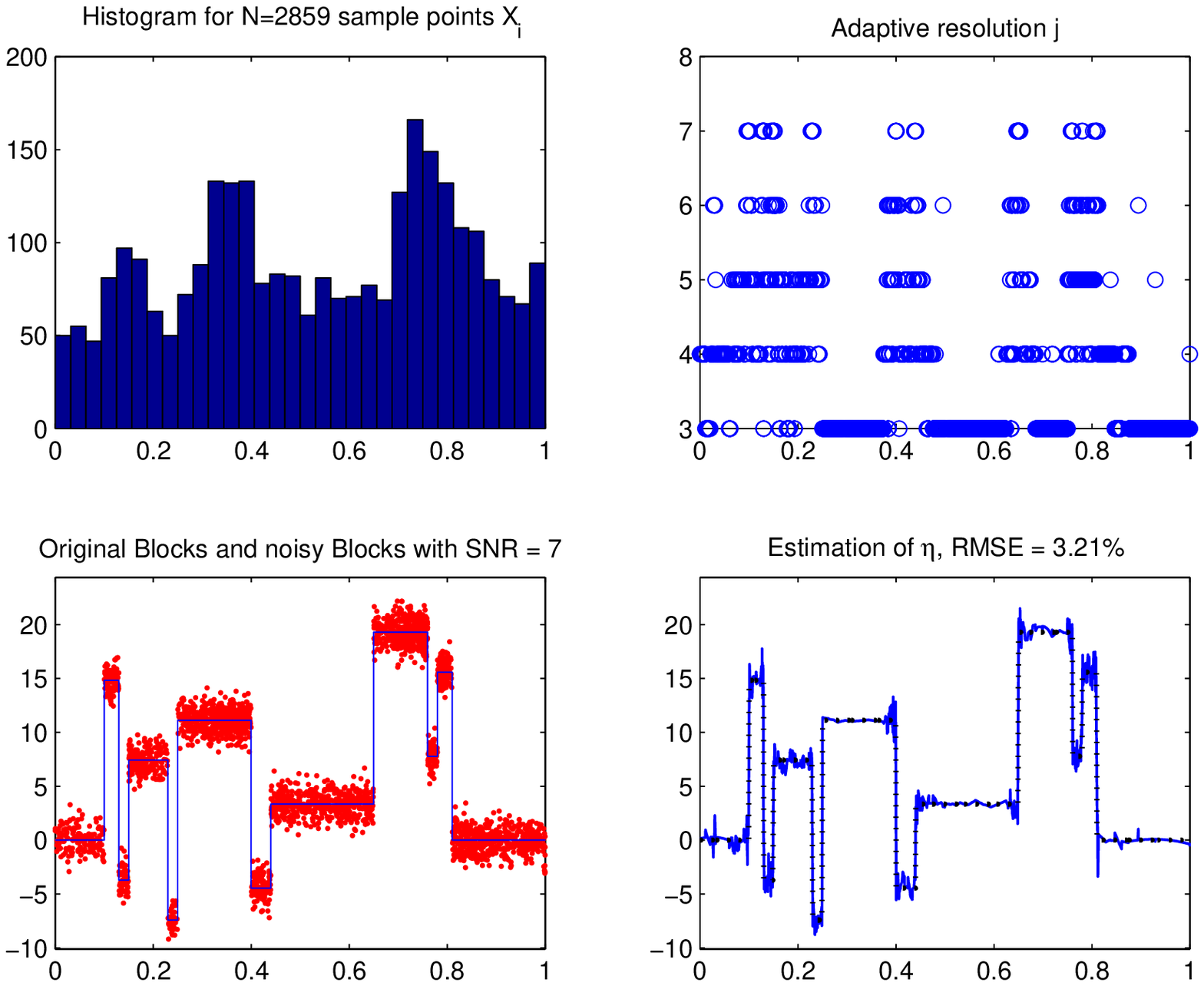}}
\end{minipage}
\end{minipage}
\end{figure}
Let us now give details about the simulation of the sample points and
the computation of the estimator. We divide the unit-interval into ten
sub-segments $A_k := 10^{-1}[k, k+1]$ for $k=0,\ldots,9$. We define the density of $X$ as follows.
\begin{align*}
\mu(x) = \sum_{k=0}^{9} p_k \lambda(A_k)^{-1} \inds_{A_k}(x).
\end{align*}
We choose the $p_k$'s at random. To that end, we denote by
$(u_k)_{0\leq k \leq 9}$ ten realizations of the uniform random
variable on $[.25,1]$, write $v = u_0 + \ldots + u_{9}$ and set $p_k
= u_k v^{-1}$. Notice that this guarantees that $ \mu \geq \min_{0
  \leq k \leq 9}10 p_k \geq \mu_{\min} = 0.25$ on $[0,1]$. We then simulate
$3000$ sample points $X_i$ according to $\mu$. Finally, we bring the
points back on the grid $G$ by assimilating them to their nearest grid
node. Since the $X_i$'s are supposed to be drawn from a law that is
absolutely continuous with respect to the Lebesgue measure on $[0,1]$,
we must keep only one data point per grid node. This reduces the
number of data points from $3000$ to the number that is reported on
top of each of the histograms.\\ 
In order to compute the adaptive estimator at sample points $X_i$, we use the boundary-corrected
scaling functions coded into Wavelab850 for $r=3$ and for which
we must have $j \geq 3$. We set $J = \ceil{\log(n / \log n)/ \log
  2}$. The elimination of redundant sample points on the
grid removes on average $150$ points so that we obtain $J = 10$. We
therefore have $\J_n = \{3,4,\ldots,10\}$. Notice interestingly that
the computation of $\eta^@_3$ requires only $8$ regressions and
$\eta^@_{10}$ requires $1,024$ of them. This is much smaller than for
the LPE whose computation necessitates as many
regressions as there are sample points at each resolution level. In
practice, we compute the
minimum eigenvalues of all regression matrices across partitions
and resolution levels and choose $\pi_n^{-1}$
to be the first decile of this set of values. When
proving theoretical results, we have chosen $\eta^@_j$ to be zero
on the small probability event where the minimum eigenvalue of the
regression matrix is smaller than $\pi_n^{-1}$. In practice we can
choose it to be an average value of the nearby cells in order to get
an estimator that is overall more appealing to the eye. In
our simulation, we in fact do not use that modification. Instead,
we modify $j^@$ to be the highest $j \in \{3,\ldots, j^@\}$ such
that $\eta^@_{j^@}$ has been computed from a valid regression matrix,
meaning a regression matrix whose
smallest eigenvalue is greater than the threshold $\pi_n^{-1}$.\\
In practice, for a given signal, we generate $\mu$ at random and
compute $\eta^@_{j^@}$ for $100$ samples drawn from $\mu$. We quantify
the performance $\eta^@_{j^@}$ by its relative RMSE,
meaning its RMSE computed at sample points $X_i$ divided by the
amplitude of the true signal, that is its maximal absolute value on
the underlying dyadic grid. We display results 
for  ``Doppler'',
``HeaviSine'', ``Bumps''
and ``Blocks'' corresponding to the
median performance among the $100$ trials. Each figure displays four graphs. Clockwise from the top
left corner, they display in turn, an histogram of sample points
$X_i$; the adaptive level $j^@$ at sample points $X_i$; the
true signal (black dots) and the estimator
$\eta^@_{j^@}$ at sample points $X_i$ (solid blue line) and its
corresponding relative RMSE in the title; and finally the original
signal (solid blue line) with its noisy
version at sample points $X_i$ (red dots).

\section{Proofs}\label{sec:proofglobal}
\subsection{Proof of the upper-bound results under \Hp{CS1}}\label{sec:proofCS1}
\subsubsection{Proof of \ref{co:main}}\label{sec:mainproof}
Consider the term
\begin{align*}
I = \int_{\A} \Exp[\abs{\eta(x) - \eta_j^@(x)}^p] \mu(x)dx.
\end{align*}
Now, apply \ref{prop:espx} and notice that $\int_{\A}
\mu(x)dx = 1$ to show that $I$ is upper-bounded by the term that
appears on the rhs of \ref{eq:ubj} stated in \ref{co:main}. In
particular, for all $1 \leq p < \infty$, we obtain $I \leq C(p)
\pi_n^p t(n)^{-p} \leq C(p) < \infty$.
This in turn proves that we can apply the Fubini-Tonelli theorem to get 
\begin{align*}
I  &= \Exp [\normL{ \eta - \eta_j^@ }{ \Lp_p( \Xs, \mu ) }^p],
\end{align*}
and concludes the proof.
\qed

\subsubsection{Proof of \ref{th:main}}
Let $x \in \A$ and $j \in \J_n$. There exists $\Hs \in
\info_j$ such that $x \in \Hs$. Let us work on the set $\{\lambda_{\min}(Q_{\Hs})
\geq \pi_n^{-1}\}$ on which $Q_{\Hs}$ is invertible. On that set, we can write
\begin{align*}
\abs{\pr_j\eta(x) - \eta_\Hs^\diamond(x) } &= \abs{ \ms{ \alpha_{\Hs} -
    \alpha^\diamond_{\Hs}}{\phi_{\Hs}(x)} }\\
&= \abs{ \ms{ Q_{\Hs}^{-1} \cdot \left( \frac{B_{\Hs}^t}{n} \cdot
      (B_{\Hs} \cdot \alpha_{\Hs} - Y_{\Hs}
      ) \right) }{\phi_{\Hs}(x)} }\\
&\leq  \normL{Q_{\Hs}^{-1}}{S} \normL{\frac{B_{\Hs}^t}{n} \cdot (B_{\Hs}
  \cdot \alpha_{\Hs} - Y_{\Hs}
      )}{\ell^2(\R^{\rr^d})} \normL{\phi_{\Hs}(x)}{\ell^2(\R^{\rr^d})}\\
&\leq \rr^{\frac d2} 2^{j\frac d2} \lambda_{\min}(Q_{\Hs})^{-1} \normL{\frac{B_{\Hs}^t}{n} \cdot(B_{\Hs}
  \cdot \alpha_{\Hs} - Y_{\Hs} )}{\ell^2(\R^{\rr^d})}.
\end{align*}
Now, notice that for all $X_i \in \Hs$, we have $Y_i = \ms{\alpha_{\Hs}}{\phi_{\Hs}(X_i)} + \re_j \eta (X_i) + \xi_i$. Write
$\re_{\Hs} = ( \re_j \eta(X_i) \inds_{\Hs}(X_i))_{1\leq i \leq n }$
and $\xi_{\Hs} = ( \xi_i \inds_{\Hs}(X_i))_{1\leq i \leq n
}$. Then, we have,
\begin{align*}
W_{\Hs} &= \abs{ \frac{B_{\Hs}^t}{n} \cdot (B_{\Hs} \cdot \alpha_{\Hs} - Y_{\Hs} )} = \abs{
\frac{B_{\Hs}^t}{n} \cdot (\xi_{\Hs} + \re_{\Hs} ) } \in \R^{\rr^d}.
\end{align*}
Thus, a direct application of \ref{up:Bernstein} allows to write, for
$ \delta > 2M 2^{-js} \linebreak[3] \max( 1 , 3\pi_n\rr^{d} \mu_{\max})$,
\begin{align*}
\Pb( \abs{ \eta(x) - \eta_\Hs^\diamond (x) }& \geq \delta, \lambda_{\min}(Q_{\Hs})
\geq \pi_n^{-1})\\
&\leq \Pb( \normL{W_\Hs}{\ell_2(\R^{\rr^d})} \geq \frac{\delta
  2^{-j\frac d2}}{2 \pi_n \rr^{\frac d2}})\\
&\leq \rr^d \sup_{k \in \Se_j(\Hs)} \Pb\left( [W_\Hs]_k \geq \frac{\delta 2^{-j\frac d2}}{2\pi_n\rr^{d}}\right)\\
&\leq \rr^d \Lambda\left(  \frac{\delta 2^{-j\frac d2}}{2 \pi_n \rr^{d}}  \right).
\end{align*}
By definition, we have $\eta^{@}_j(x) =
\eta_{\Hs}^{@}(x)$, so that we have
\begin{align}\label{eq:split}
\begin{split}
\Pb( \abs{\eta(x) - \eta^@_j(x)} \geq \delta )
&=  \Pb( \abs{\eta(x) - \eta^@_\Hs(x)} \geq \delta,
\lambda_{\min}(Q_{\Hs}) \geq \pi_n^{-1} )\\
&+ \Pb(\abs{\eta(x) -
  \eta^@_\Hs(x)} \geq \delta, \lambda_{\min}(Q_{\Hs}) < \pi_n^{-1}).
\end{split}
\end{align}
By construction, $\eta_\Hs^@(x) =
\eta_\Hs^\diamond(x)$ on the event $\{ \lambda_{\min}(Q_{\Hs}) \geq \pi_n^{-1} \}$ and
$\eta_\Hs^@(x) = 0$ on its complement. So that we obtain $\abs{\eta(x) -
  \eta^@_\Hs(x)} = \abs{\eta(x)} \leq M$ on the rhs of
\ref{eq:split}. Notice in addition that $M 2^{-js} \geq
\abs{\re_j\eta(x)}$ under \Hp{H_s^r} (see Appendix). Finally, we obtain, for $\tfrac{\delta}2
> M 2^{-js} \geq \abs{\re_j\eta(x)}$,
\begin{align*}
\Pb( &\abs{\eta(x) - \eta^@_j(x)} \geq \delta )\\
&\leq  \Pb( \abs{\pr_j\eta(x) - \eta^\diamond_\Hs(x)} \geq \frac{\delta}2,
\lambda_{\min}(Q_{\Hs}) \geq \pi_n^{-1} ) +
\Pb(\lambda_{\min}(Q_{\Hs}) < \pi_n^{-1})\ind{ \bar M \geq \delta},
\end{align*}
where we have written $\bar M = M$. The term on the lhs has been dealt with above. The term on the rhs is
tackled using \ref{prop:pbvp}. This concludes the proof. 
\begin{remark}\label{rem:updateO2}
Under \Hp{O2}, we have $\abs{\eta^@_{\Hs}(x)}\leq M$, and since $\eta \in
\Lip^s(\Xs,M)$, we obtain $\abs{\eta(x) -
  \eta^@_\Hs(x)} \leq 2M$ on the rhs of \ref{eq:split}. While on the lhs, it
is straightforward that (see \mycite[Chap.~10]{Gyorfi2001}) 
\begin{align*}
\abs{\eta(x) - \eta^@_\Hs(x)} & = \abs{\eta(x) -
  T_M(\eta^\diamond_\Hs(x))} \leq \abs{\eta(x) - \eta^\diamond_\Hs(x)}.
\end{align*}
Under \Hp{O1}, the proof remains unchanged. So that the proof still
holds with 
\begin{align*}
\bar M &=
\begin{dcases}
2M, &\text{under \Hp{O2}},\\
M, & \text{otherwise}.
\end{dcases}
\end{align*}
\end{remark}
\qed

\subsubsection{Proof of \ref{th:adaptation} }\label{sec:proofadaptation}
This result is obtained after a slight modification of
\mycite[Proposition~3.4]{Lepski1997}. In the same way as in the proof of
\ref{th:main}, we are brought back to controlling $\Exp \abs{
  \eta_{j^@(x)}^@(x) - \eta(x) }^p$ for all $x \in \A$. To that end,
we split this term as follows
\begin{align*}
\Exp \abs{
  \eta_{j^@(x)}^@(x) - \eta(x) }^p &= \Exp \abs{
  \eta_{j^@(x)}^@(x) - \eta(x) }^p (\ind{ j^@(x) \leq j_s} + \ind{ j^@(x) > j_s})\\
&= I + II.
\end{align*}
Let us first deal with $I$. Notice that
\begin{align*}
2^{1-p}\abs{ \eta^@_{j^@(x)}(x) - \eta(x)}^p &\leq \abs{ \eta^@_{j^@(x)}(x) -
  \eta^@_{j_s}(x) }^p + \abs{ \eta^@_{j_s}(x) - \eta(x) }^p.
\end{align*}
The last term is of the good order since 
\begin{align*}
\Exp\abs{\eta^@_{j_s}(x) - \eta(x)}^p &\leq
C(p) \pi_n^p \max\left( 2^{-j_s s}, \frac{2^{j_s \frac
      d2}}{\sqrt n}\right)^p\\ 
&= \frac{C(p)}{(\kappa \log n)^{\frac p2}}
\left( t(n) 2^rn^{-\frac s{2s+d}}\right)^p,
\end{align*}
according to \ref{prop:espx}, \ref{eq:sa} and \ref{eq:sc}. Regarding the first term, notice that on the event $\{ j^@(x) \leq
j_s  \}$, one has got
\begin{align*}
\abs{\eta^@_{j^@(x)}(x) -
  \eta^@_{j_s}(x)} &\leq g(j^@(x), j_s) \leq \sup_{j_{\varpi} \leq k \leq j_s} g(k, j_s)\\
&\leq g(j_s,j_s) =  2 t(n) \frac{2^{j_s \frac d2}}{\sqrt n} \leq
2t(n)2^r n^{-\frac s{2s+d}},
\end{align*} 
where we have used \ref{eq:sa} and \ref{eq:sc} and which is of the good order too. Let us now turn to $II$. For any two
$j < k$, we write
\begin{align*}
\Gr(x,j,k) = \{ \abs{ \eta^@_j(x) - \eta^@_k(x)} > g(j,k) \}.
\end{align*}
Write $\J_n(j) = \{ k \in \J_n: k > j \}$. Notice first that we have the following inclusions
\begin{align*}
\{ j^@(x) = j \} &\subseteq \bigcup_{ k \in \J_n(j-1)} \Gr(x,j-1,k),\\
\{ j^@(x) > j_s \} &= \bigcup_{j \in \J_n(j_s)} \{ j^@(x) = j\}
\subseteq \bigcup_{j \in \J_n(j_s)} \bigcup_{k \in \J_n(j-1)} \Gr(x,j-1,k).
\end{align*}
Therefore, we can write
\begin{align*}
II &\leq \sum_{j \in \J_n(j_s)} \Exp \abs{ \eta^@_{j^@(x)}(x) -
  \eta(x)}^p \ind{j^@(x) = j}\\
&\leq \sum_{j \in \J_n(j_s)} \sum_{k \in \J_n(j-1) } \Exp \abs{ \eta^@_{j}(x) -
  \eta(x)}^p \inds_{\Gr(x,j-1,k)}.
\end{align*}
Now, we notice that
\begin{align*}
\abs{\eta_{j}^@(x) - \eta_k^@(x)} &\leq \abs{\eta_{j}^@(x) -
  \eta(x)} + \abs{ \eta(x) - \eta_k^@(x)}.
\end{align*} 
So that 
\begin{align*}
\Gr(x , j, k) &= \{ \abs{\eta_{j}^@(x) - \eta_k^@(x)} > g(j,k)\}\\
& \subset \left\{\abs{\eta_{j}^@(x) - \eta(x)} >  \frac{2^{j \frac
    d2}}{\sqrt n} t(n) \right\} \bigcup \left\{\abs{\eta_{k}^@(x) - \eta(x)} >
 \frac{2^{k \frac
    d2}}{\sqrt n} t(n) \right\},\\      
\Pb(\Gr(x , j, k)) &\leq \Pb\left( \abs{\eta_{j}^@(x) - \eta(x)} > \frac{2^{j \frac
    d2}}{\sqrt n} t(n) \right) + \Pb\left( \abs{\eta_{k}^@(x) - \eta(x)} >  \frac{2^{k \frac
    d2}}{\sqrt n} t(n) \right).
\end{align*}
So that a direct application of the Cauchy-Schwarz inequality leads to
\begin{align*}
\Exp \abs{ \eta^@_{j}(x) - \eta(x)}^p \inds_{\Gr(x,j-1,k)}  &\leq (\Exp \abs{ \eta^@_{j}(x) -
  \eta(x)}^{2p})^{\frac 12} \Pb(\Gr(x,j-1,k))^{\frac 12}.
\end{align*}
Now, a direct application of \ref{prop:espx} for $j_s \leq j\leq J$ gets us
\begin{align*}
(\Exp \abs{ \eta^@_{j}(x) -
  \eta(x)}^{2p})^{\frac 12} \leq  \sqrt{C(2p)} \pi_n^p \max\left( 2^{-js}, \frac{2^{j\frac
    d2}}{\sqrt n} \right)^p \leq \sqrt{C(2p)}(\kappa \log n)^{-\frac p2}.
\end{align*}
Besides, notice that for $j_s \leq j < k \leq J$, we can apply
\ref{prop:expbound} with $\kappa \geq \tfrac p2 C_9^{-1}$ to obtain
\begin{align*}
\Pb\left( \abs{\eta_{j}^@(x) - \eta(x)} >  \frac{2^{j \frac
    d2}}{\sqrt n} t(n) \right) \vee \Pb\left( \abs{\eta_{k}^@(x) -
  \eta(x)} >  \frac{2^{k \frac
    d2}}{\sqrt n} t(n) \right) \leq 5\rr^{2d} n^{-\frac p2}.
\end{align*}
To conclude the
proof, it remains to notice
that $\#\J_n \leq \log n$ and remark that the multiplicative constant in the
upper-bound of \ref{th:adaptation} is indeed smaller than, say, $5$ for $n$ large enough. \qed

\subsubsection{A few useful Propositions and Lemmas}\label{sec:usefulprop}

\begin{proposition}\label{prop:espx}
Fix $r \in \N$ and assume we are under \Hp{CS1} and \Hp{H_s^r}. Then, For any $x
\in \A$ and $j \in \J_n$, one has got
\begin{align*}
\Exp[\abs{\eta(x) - \eta_j^@(x)}^p] \leq C(p) \pi_n^p \max\left(2^{-js},
\frac{2^{j\frac d2}}{\sqrt n}\right)^p,
\end{align*}
where 
\begin{align*}
C(p) =
3^pM^p\max(1, \rr^{2d} \mu_{\max})^p +
C_5(r,d,p,\mu_{\max};K, \sigma) + 2 M^p\rr^{2d},
\end{align*}
and $C_5$ is made explicit in the proof at \ref{eq:const}.
\end{proposition}
\begin{proof}
For any $x \in \A$, take $\delta = 3 M 2^{-js} \max(1,3\pi_n \rr^{d}
\mu_{\max})$. Notice first that $\max(1,3\pi_n \rr^{d}
\mu_{\max}) \leq \pi_n \max(1,3\rr^{d}\mu_{\max})$ since, by construction, $\pi_n^{-1}
\leq 1$ in any case. Now, write
\begin{align*}
\Exp[\abs{\eta(x) - \eta_j^@(x)}^p]&= \int_{\R^+} pt^{p-1} \Pb(
\abs{\eta(x) - \eta_j^@(x)} \geq t) dt\\
&\leq \delta^p + \int_{\delta}^{+\infty} pt^{p-1}
\Pb(\abs{\eta(x) - \eta_j^@(x)} \geq t)dt.
\end{align*}
As $\delta$ has been fixed, we only need to tackle the rhs above,
which we will denote by $II$. Using
\ref{th:main}, we can write
\begin{align*}
II &\leq 2\rr^{2d} \exp\left( - n2^{-jd} \frac{\pi_n^{-2}}{
    2\mu_{\max}\rr^{4d} + \frac 43 \rr^{2d} \pi_n^{-1}} \right )
\int_{0}^{M} pt^{p-1}dt \\
 &+ \rr^d \int_0^{\infty} pt^{p-1} \Lambda\left(  \frac{t 2^{-j\frac
       d2}}{2 \pi_n\rr^{d}}  \right) dt.
\end{align*} 
Denote by $II_1$ and $II_2$ the lhs and rhs terms above, respectively. Now, recall that $j \leq
J$, where $2^{Jd} \leq n t(n)^{-2}$ and $t(n)^{2} = \kappa \pi_n^2 \log
n$. Therefore, as soon as
\begin{align*}
\kappa &\geq \frac p2 \left(2\mu_{\max}\rr^{4d} + \frac 43 \rr^{2d} \pi_n^{-1}\right),
\end{align*}
we have $II_1 \leq 2 M^p \rr^{2d} n^{-\frac p2}$. Let us now turn
to $II_2$. Assume first that we are working under the bounded noise
assumption, \Hp{N1}. In that case, we have
\begin{align*}
II_2 &\leq 2\rr^d \int_0^{\infty} pt^{p-1}\exp \left( - \frac{n
    2^{-jd} t^2 \pi_n^{-2} }{ 64 K^2 \rr^{2d} \mu_{\max} + 8 K\rr^{d} \pi_n^{-1}t} \right) dt\\
&\leq  C_2(r,d,p,\mu_{\max},K) \left( \pi_n \frac{2^{j\frac d2}}{\sqrt n} \right)^p.
\end{align*}
where the last inequality results from the change of variable $ u =
\sqrt{n}2^{-j\frac d2} \pi_n^{-1} t$ together with the fact that
$2^{jd} \leq n$
and we have written
\begin{align*}
C_2 := 2\rr^d \int_0^{\infty} pt^{p-1}\exp \left( - \frac{ t^2}{ 64 K^2 \rr^{2d} \mu_{\max} + 8 K\rr^{d} t } \right) dt.
\end{align*}
Assume now that we are working under the Gaussian noise assumption
\Hp{N2}. In that case, we have
\begin{align*}
II_2 &\leq \rr^d \int_0^{\infty} pt^{p-1} \Biggl( 1 \wedge \Biggl\{
    \frac{2 \sigma \rr^{\frac d2}(4 \rr^{d}\mu_{\max} + 2t\pi_n^{-1} )^{\frac
        12}}{t \pi_n^{-1} 2^{-j\frac d2} \sqrt{2\pi n}}\\
 &\exp\left(-
    \frac{n 2^{-jd} \pi_n^{-2} t^2 \sigma^{-2}}{ 4 \rr^{2d}\mu_{\max} +
      2\rr^d \pi_n^{-1}t }  \right)
\Biggr \} \Biggr) dt\\
 &+ 2\rr^d \int_0^{\infty} pt^{p-1} \exp\left( -\frac{n 2^{-jd}
     \pi_n^{-2} t^2}{ 8 \rr^{2d} \mu_{\max}  +
    \frac {8}3 \rr^{d} \pi_n^{-1} t }  \right) dt.
\end{align*}
Denote by $II_3$ and $II_4$ the first and second term,
respectively. They can both be handled in the exact same way as $II_2$, which
leads to
\begin{align*}
II_4 &\leq C_4(r,d,p,\mu_{\max} )  \left( \pi_n \frac{2^{j\frac d2}}{\sqrt n} \right)^p,
\end{align*}
where we have written
\begin{align*}
C_4 := 2\rr^d \int_0^{\infty} pt^{p-1} \exp\left( -\frac{t^2}{ 8 \rr^{2d} \mu_{\max}  +
    \frac {8}3 \rr^{d} t }  \right) dt,
\end{align*}
and 
\begin{align*}
II_3 &\leq C_3(r,d,p,\mu_{\max}, \sigma) \left( \pi_n \frac{2^{j\frac d2}}{\sqrt n} \right)^p,
\end{align*}
where we have written
\begin{align*}
C_3 &:= \rr^d \int_0^{\infty} pt^{p-1} \Biggl( 1 \wedge \Biggl\{
    \frac{2\sigma\rr^{\frac d2}(4\rr^{d}\mu_{\max} + 2t)^{\frac
        12}}{t\sqrt{2\pi}}\\
 &\exp\left(-
    \frac{t^2 \sigma^{-2}}{ 4 \rr^{2d}\mu_{\max} +
      2\rr^d t }  \right) \Biggr \} \Biggr) dt.
\end{align*}
To conclude, let us write 
\begin{align}\label{eq:const}
C_5(r,d,p,\mu_{\max};K, \sigma) =
\begin{dcases} 
C_2(r,d,p,\mu_{\max},K) &\text{under \Hp{N1}}\\
C_3(r,d,p,\mu_{\max}, \sigma) + C_4(r,d,p,\mu_{\max})
&\text{under \Hp{N2}}
\end{dcases}
\end{align}
Therefore, we ultimately obtain
\begin{align*}
\Exp[\abs{\eta(x) - \eta_j^@(x)}^p]& \leq \left(
3^pM^p\max(1, 3\rr^{d} \mu_{\max})^p +
C_5 + 2 M^p\rr^{2d}
\right)\\
& \pi_n^p \max\left( 2^{-js}, \frac{2^{j\frac d2}}{\sqrt n} \right)^p,
\end{align*}
which concludes the proof.
\end{proof}

\begin{proposition}\label{prop:expbound}
Fix $r$ in $\N$ and assume we are under \Hp{CS1} and \Hp{H_s^r}. This
means in particular that
$s \in (0,r)$. Let $j $ be such that $j_s \leq j \leq J$. Let $t(n)^2= \kappa
\pi_n^2 \log n$, and define
\begin{align*}
C_9(r, d, \mu_{\max}, \pi_n; K, \sigma) :=
\begin{cases} 
C_6(r,d, \mu_{\max}, K, \pi_n), &\text{under \Hp{N1}}\\
C_6(r,d, \mu_{\max}, \sigma, \pi_n), &\text{under \Hp{N2}}
\end{cases},
\end{align*}
where $C_6$ is defined in \ref{eq:defCsix} below. Then we have, for $n$ large enough,
\begin{align*}
\Pb\left( \abs{\eta_{j}^@(x) - \eta(x)} > \frac{2^{j \frac
    d2}}{\sqrt n} t(n) \right) \leq  5 \rr^{2d} n^{-\kappa C_9}.
\end{align*}
\end{proposition}
\begin{proof}
The proof relies on a direct application of \ref{th:main}. Write $C_0 = 2M\max(1, 3\pi_n\rr^{d}\mu_{\max})$ and notice indeed that the theorem
applies since for $j \geq j_s$, we get $2^{j\frac d2}n^{-\frac 12} \geq
2^{-(r+\frac d2)} 2^{-j_s s}$ (see \ref{eq:sb}) and, as soon as $n$ is
large enough, we have $t(n) \geq 2^{r+\frac d2}C_0$. This leads us to
\begin{align*}
\Pb\left( \abs{\eta_{j}^@(x) - \eta(x)} >  \frac{2^{j \frac
    d2}}{\sqrt n} t(n) \right)&\leq  2\rr^{2d} \exp\left( - n2^{-jd} \frac{\pi_n^{-2}}{
    2\mu_{\max}\rr^{4d} + \frac {4}3 \rr^{2d} \pi_n^{-1}} \right )\\
 &+ \rr^{d} \Lambda\left(  \frac{ t(n) }{2 \pi_n \rr^{d}\sqrt{n}}  \right).
\end{align*}
Let us denote the first term by $I$ and the second one by $II$. $I$ is
easily tackled noticing that for $j \leq J$, $n2^{-jd} \geq n2^{-Jd} \geq
t(n)^2 = \kappa \pi_n^2 \log n$. So that, we obtain $I \leq 2\rr^{2d}
n^{-\kappa C_6}$, where we have written
\begin{align}\label{eq:defCsix}
C_6(r,d, \mu_{\max}, K, \pi_n) := \frac {\min(1,K^{-2})}{ 64\mu_{\max}\rr^{2d} +
  8 \rr^{d} \pi_n^{-1} }.
\end{align}
Let us now turn to
$II$. Assume first we work under \Hp{N1}. Then we can write
\begin{align*}
II &\leq 2\rr^{d}\exp \left( - \frac{
    t(n)^2  \pi_n^{-2}}{64 \rr^{2d} K^2 \mu_{\max} + 8 \rr^{d}  K  \pi_n^{-1} \frac{2^{j\frac d2}t(n)}{\sqrt
      n}} \right). 
\end{align*}
Notice first that $ 2^{j\frac d2} t(n) \leq \sqrt{n}$. Therefore, we
obtain $II \leq 2\rr^d n^{-\kappa C_6}$. Assume
now that we work under \Hp{N2}. In that case, we obtain
\begin{align*}
II &\leq \rr^{d} \Biggl( 1 \wedge \Biggl\{
  \frac{2\rr^{\frac d2} \sigma (4\rr^{d}\mu_{\max} + 2\pi_n^{-1}\frac{2^{j\frac d2}t(n)}{\sqrt n}
   )^{\frac 12}}{t(n)\pi_n^{-1}\sqrt{2\pi}} \\
&\exp\left(-
    \frac{ t(n)^2\pi_n^{-2} \sigma^{-2}}{4\rr^{2d} \mu_{\max} + 2 \rr^{d}
       \pi_n^{-1}\frac{2^{j\frac d2} t(n)}{\sqrt n}}  \right) \Biggr \}\Biggr)\\
&+ 2 \rr^{d} \exp\left( -\frac{ t(n)^2 \pi_n^{-2}}{8\rr^{2d}\mu_{\max}  +
    \frac {8}3 \rr^{d} \pi_n^{-1}\frac{2^{j\frac d2} t(n)}{\sqrt n} }  \right).
\end{align*}
We proceed exactly as under \Hp{N1}. So that we obtain $II \leq C_7
n^{-\kappa C_8}$, where  
\begin{align*}
C_8(r, d, \mu_{\max}, \sigma, \pi_n) &:= \frac {\min(1, \sigma^{-2})}{ 4
  \rr^{2d} \mu_{\max} + 2 \rr^{d} \pi_n^{-1}},\\
C_7(r,d,\mu_{\max}, \sigma, \pi_n, t(n)) &= \rr^{d}
\frac{2\rr^{\frac d2}\sigma (4\rr^{d}\mu_{\max} + 2\pi_n^{-1})^{\frac
      12}}{t(n) \pi_n^{-1} \sqrt{2\pi}}  + 2 \rr^{d} .
\end{align*}
So that $C_7 \leq 3 \rr^{d}$ for $n$ large enough. Notice finally that
$C_8(r,d,\mu_{\max}, t, \pi_n) \geq  C_6(r,d,\mu_{\max}, t, \pi_n)$. This
concludes the proof.
\end{proof}

\begin{proposition}\label{prop:pbvp}
Fix an integer $r \geq 1$ and assume we are under \Hp{CS1}. Let $x \in
\A$ and $j \in \J_n$. By construction, there exists $\Hs \in \info_j$
such that $x \in \Hs$. Recall besides that  $\#\Se_j(\Hs) = \rr^d$,
where $\rr =2r-1$ is obviously independent of both $x$
and $j$. Write $\norm{.}=\normL{.}{\ell_2(\R^{\rr^d})}$ and assume there exists a
strictly positive constant $g_{\min}$ independent of $x$ and $j$ such that
\begin{align}\label{eq:ub}
\lambda_{\min}(\Exp Q_{\Hs}) = \min_{u \in \R^{\rr^d}: \norm{u} = 1} \ms{u}{\Exp Q_\Hs u} \geq g_{\min}.
\end{align}
Then, for any real number $t$ such that $0 < t \leq
\tfrac{g_{\min}}2$, we have 
\begin{align*}
\Pb( \lambda_{\min}(Q_{\Hs}) \leq t) &\leq 2\rr^{2d} \exp\left( - n2^{-jd} \frac{t^2}{
    2\mu_{\max}\rr^{4d} + \frac 43 \rr^{2d} t} \right ).
\end{align*} 
\end{proposition}
\begin{proof}
Under the assumption described in \ref{eq:ub}, we get
\begin{align*}
\lambda_{\min}(Q_{\Hs}) &\geq \min_{u \in \R^{\rr^d}: \norm{u} = 1}
\ms{u}{\Exp Q_{\Hs} u} + \min_{u \in \R^{\rr^d}: \norm{u}=1} \ms{u}{ (Q_{\Hs} - \Exp Q_{\Hs}) u}\\
&\geq 2t  - \sum_{\nu, \nu' \in \Se_j(\Hs)} \abs{  [Q_{\Hs}]_{\nu, \nu'} - [\Exp Q_{\Hs}]_{\nu, \nu'}}.
\end{align*}
Write $T_i = \phi_\jn(X_i)\phi_{j,\nu'}(X_i)\inds_{\Hs}(X_i)  - \Exp \phi_\jn(X)\phi_{j,\nu'}(X)\inds_{\Hs}(X)$,
so that $\Exp T_i = 0$, $\Var T_i \leq \mu_{\max}2^{jd}$ and $\abs{T_i} \leq
2^{jd+1}$. A direct application of Bernstein inequality for any
$\delta >0$ leads to
\begin{align*}
&\Pb( \abs{  [Q_{\Hs}]_{\nu, \nu'} - [\Exp Q_{\Hs}]_{\nu, \nu'}} \geq
\delta )\\
&= \Pb( \abs{  \frac 1n \sum_{i=1}^n \phi_\jn(X_i)\phi_{j,\nu'}(X_i)\inds_{\Hs}(X_i) - \Exp \phi_\jn(X)\phi_{j,\nu'}(X)\inds_{\Hs}(X) }
\geq \delta )\\
&\leq 2\exp\left( - \frac{n2^{-jd} \delta^2}{2\mu_{\max} + \frac 43 \delta } \right).
\end{align*}
To conclude, we write
\begin{align*}
\Pb( \lambda_{\min}(Q_{\Hs}) \leq t ) &\leq \Pb( \sum_{\nu, \nu' \in
  \Se_j(\Hs)} \abs{  [Q_{\Hs}]_{\nu, \nu'} - [\Exp Q_{\Hs}]_{\nu, \nu'}} \geq t ) \\
&\leq 2\rr^{2d} \exp\left( - n2^{-jd} \frac{t^2}{
    2\mu_{\max}\rr^{4d} + \frac 43 \rr^{2d} t} \right ).
\end{align*}
\end{proof}

\begin{proposition}\label{prop:support}
Fix an integer $r \geq 1$ and assume we are under \Hp{CS1}. For any
$x \in \A$ and $j \in \J_n$, we denote by $\Hs$ the
unique hypercube of $\info_j$ such that $x \in \Hs$. Then, there exists a strictly
positive absolute constant $g_{\min}$ independent of both $x$ and
$j$ such that, for all $j \in \J_n$
and all $x \in \A$, we have $\lambda_{\min}(\Exp Q_{\Hs}) \geq g_{\min} >
0$.
\end{proposition}
\begin{proof}
For any $u \in
\R^{\rr^d}$ such that $\normL{u}{\ell_2(\R^{\rr^d})}=1$, we can write
\begin{align}
\ms{u}{\Exp Q_{\Hs}\cdot u} &= \int_{\A} \left(\sum_{\nu \in \Se_j(\Hs)}
  u_{\nu} \phi_\jn(w) \inds_{\Hs}(w) \right)^2 \mu(w) dw \displaybreak[2]\nonumber\\
& \geq \mu_{\min} \int_{ \Hs } \left( \sum_{\nu \in
    \Se_j(\Hs)} u_\nu \phi_\jn(w) \right)^2 dw,\displaybreak[2]\label{eq:lbCS1integral}\\
&= \mu_{\min} \int_{ [0,1]^d } \left( \sum_{\nu \in
    \fdel} u_\nu \phi_{\nu}(w) \right)^2 dw,\label{eq:lbCS1}
\end{align}
where $\fdel$ has been defined in \ref{eq:fdeldef} and the last
equality results from the fact that the value of the integral on the
rhs of \ref{eq:lbCS1integral} is invariant with $\Hs$. Let us denote
by $\Sp^{\rr^d-1}$ the unit-sphere of $\R^{\rr^d}$. As detailed in \mycite{Monnier2011}, the map 
\begin{align*}
u \in \Sp^{\rr^d-1} \mapsto \int_{ [0,1]^d } \left( \sum_{\nu \in
    \fdel} u_\nu \phi_{\nu}(w) \right)^2 dw,
\end{align*}  
is absolutely continuous with respect to $u$ on the compact subset
$\Sp^{\rr^d-1}$ of $\R^{\rr^d}$. It therefore reaches its minimum at some point $u^*
\in \Sp^{\rr^d-1}$. It is a direct consequence of the \textbf{local linear
independence property} of the scaling functions $(\phi_k)$ (see \ref{prop:loclinind}) that
\begin{align*}
\int_{ [0,1]^d } \left( \sum_{\nu \in \fdel} u^*_\nu \phi_{\nu}(w) \right)^2 dw = c_{\min} > 0,
\end{align*}
where $c_{\min}$ is a constant that is both independent from $x$ and
$j$. This concludes the proof with $g_{\min} = \mu_{\min}c_{\min}$.
\end{proof}

\begin{proposition}\label{up:Bernstein}
Let $(X_i)_{i=1,\ldots,n}$ and $(\xi_i)_{i=1,\ldots,n}$ be sequences of independent random
variables such that $\Exp(\xi\vert X) =0$. Take any $j\geq
j_r$. Moreover, assume we are given a
function $\re_j(.)$ such that $\normL{\re_j(.)}{\Lp_{\infty}(\Xs,\lambda)}
\leq M 2^{-js}$, a subset $\Hs$ of $\Xs$ and a scaling function $\phi_\jk$. Write
\begin{align*}
W_\jk &= \frac 1n \sum_{i=1}^n \phi_\jk(X_i)\inds_{\Hs}(X_i)(\re_j(X_i) + \xi_i),
\end{align*}
and define 
\begin{align*}
\Lambda(\delta ) = 
\begin{dcases}
2\exp \left( - \frac{n
    \delta^2 }{18K^2 \mu_{\max} + 4 K 2^{j\frac d2}\delta}
\right), &\text{under \Hp{N1}}\\
\\
1 \wedge \left\{ \frac{2 \sigma (\mu_{\max} + 2^{j\frac d2}
    \delta)^{\frac 12}}{\delta \sqrt{2\pi n}} \exp\left(-
    \frac{n \delta^2 \sigma^{-2}}{\mu_{\max} + 2^{j\frac d2} \delta}  \right)
\right \}\\
\qquad + 2 \exp\left( -\frac{n \delta^2}{2\mu_{\max}  +
    \frac 43 2^{j\frac d2}\delta }  \right),  &\text{under \Hp{N2}}
\end{dcases}
\end{align*}
Then, for all $\delta > 3\mu_{\max} M 2^{-j(s+\frac d2)}$, we have
\begin{align*}
\Pb( \abs{W_\jk} \geq \delta ) &\leq \Lambda(\delta).
\end{align*}
\end{proposition}
\begin{proof}
Notice indeed that
\begin{align*}
W_\jk&\leq \abs{\frac 1{n} \sum_{i=1}^{n} \phi_\jk(
  X_i) \xi_i \inds_{\Hs}(X_i)}\\ 
&+ \abs{\frac 1{n} \sum_{i=1}^{n} \phi_\jk(
  X_i) \re_j(X_i)\inds_{\Hs}(X_i) - \Exp \phi_\jk(X) \re_j(X) \inds_{\Hs}(X)}\\
 &+ \abs{\Exp \phi_\jk(X) \re_j(X)\inds_{\Hs}(X) }  \\
&= I + II + III.
\end{align*}
So that we can write
\begin{align*}
\Pb(\abs{W_\jk} \geq \delta) \leq \Pb(I \geq \delta/3)+\Pb(II \geq \delta/3)+\Pb(III \geq \delta/3).
\end{align*}
Now it is enough to notice that 
\begin{align*}
III &\leq \int \abs{\phi_\jk(w) \re_j(w)} \inds_{\Hs}(w) \mu(w) dw \\
 &\leq \mu_{\max} \int_{\Xs} \abs{\phi_\jk(w) \re_j(w)} dw \\
&\leq \mu_{\max} \normL{\phi_\jk}{\Lp_1(\Xs,\lambda)}
\normL{\re_j}{\Lp_{\infty}(\Xs,\lambda)}\\
&\leq \mu_{\max} M 2^{-j(s + \frac d2)}. 
\end{align*}
So that $\Pb( III \geq \delta/3) = 0$ as soon as
$\delta > 3\mu_{\max} M 2^{-j(s+\frac d2)}$.\\
Now, turn to $II$ and write $II = \abs{ \sum T_i /n}$ with $T_i = \phi_\jk(
  X_i) \re_j(X_i)\inds_{\Hs}(X_i)\linebreak[4] - \Exp \phi_\jk(X) \re_j(X)\inds_{\Hs}(X)$. Obviously $\Exp T_i = 0$,
  $\Var T_i \leq \Exp (\phi_\jn(X) \re_j(X) \inds_{\Hs}(X))^2 \linebreak[4]\leq \mu_{\max} M^2 2^{-2js}$ and
  $\abs{T_i} \leq M 2^{-js} 2^{j\frac d2 +1 }$. So that we can apply
  Bernstein inequality to get
\begin{align*}
\Pb( II \geq \delta/3) & \leq 2 \exp\left(- \frac{n
    2^{2js} \delta^2}{ 18\mu_{\max}M^2 + 4 M 2^{j \frac d2} 2^{js} \delta}  \right).
\end{align*}
And finally, turn to $III$. Assume first that the noise $\xi$ is
bounded by $K$. We have obviously $\Exp \phi_\jk(
  X_i) \xi_i \inds_{\Hs}(X_i) = 0$, $\Var( \phi_\jk(
  X_i) \xi_i \inds_{\Hs}(X_i) ) \leq K^2 \mu_{\max}$ and $\abs{\phi_\jk(
  X_i) \xi_i \inds_{\Hs}(X_i) } \leq K 2^{j\frac d2 +1}$, so that
\begin{align*}
\Pb( I \geq \delta/3) & \leq 2 \exp \left( - \frac{n
    \delta^2 }{ 18 K^2\mu_{\max} + 4 K 2^{j\frac d2}\delta } \right).
\end{align*}
Now, it is enough to notice that for all $s > 0$ and $j$ such that
$j \geq \tfrac 1s \log_2 \frac{M}{K}$ (which becomes a constraint for $K< M$ only), 
\begin{align*}
\frac{n 2^{2js}\delta^2}{ 18\mu_{\max}M^2 + 4 M 2^{j \frac d2} 2^{j
    s} \delta} &\geq \frac{n
    \delta^2 }{ 18K^2 \mu_{\max} + 4 K 2^{j\frac d2}\delta },
\end{align*}
which concludes the proof under \Hp{N1}. When $j \geq \tfrac 1s \log_2 3M$, the conclusion under \Hp{N2}
is a direct consequence of \ref{prop:gbound}.
\end{proof}

\begin{proposition}\label{prop:gbound}
Let $\phi_\jk$ be a scaling function and $\Hs$ a subset of $\Xs$. Define 
\begin{align*}
I = \frac 1n \sum_{i=1}^n \phi_\jk(X_i)\xi_i \inds_{\Hs}(X_i).
\end{align*}
Assume now that the noise $\xi$ is
conditionally Gaussian, that is we are under \Hp{N2}. Then, we notice
that, conditionally on $X_1, \ldots, X_n$, $I \sim \Phi(0,
\sigma \rho_\jk/\sqrt n)$, where $\rho_\jk^2 =  n^{-1}\sum_{i=1}^n
\phi_\jk(X_i)^2\inds_{\Hs}(X_i)$. Then, for all $\delta >0$, one can write
\begin{align*}
\Pb(\abs{I} \geq \delta ) &\leq  1 \wedge \left\{ \frac{2 \sigma(\mu_{\max} + 2^{j\frac d2}
    \delta)^{\frac 12}}{\delta \sqrt{2\pi n}} \exp\left(-
    \frac{n \delta^2 \sigma^{-2}}{\mu_{\max} + 2^{j\frac d2} \delta}  \right)
\right \}\\
&+ 2 \exp\left( - \frac{n \delta^2}{ 2\mu_{\max}  +
    \frac 43 2^{j \frac d2} \delta }  \right).
\end{align*}
\end{proposition}
\begin{proof}
For any $\delta>0$, we write
\begin{align*}
\Cd_\jk(\delta) = \left\{ \abs{ \frac 1n \sum_{i=1}^n
    \phi_\jk(X_i)^2\inds_{\Hs}(X_i) - \Exp \phi_\jk(X)^2\inds_{\Hs}(X) } \leq \delta
\right \}.
\end{align*}
Notice first that
\begin{align*}
\Cd_\jk(2^{j\frac d2} \delta ) \subset \{ \rho_\jk^2 \leq \mu_{\max} + 2^{j\frac d2} \delta \}.
\end{align*}
So that
\begin{align*}
\ind{ \abs{I} \geq \delta } &\leq  \ind{\abs{I} \geq \delta }
\ind{\rho_\jk^2 \leq \mu_{\max} + 2^{j\frac d2} \delta } +
\ind{\abs{I} \geq \delta } \inds_{\Cd_\jk^c( 2^{j\frac d2}\delta)}\\
&\leq  \ind{\abs{I} \geq \delta }
\ind{\rho_\jk^2 \leq \mu_{\max} + 2^{j\frac d2} \delta } +
\inds_{\Cd_\jk^c( 2^{j\frac d2}\delta)}.
\end{align*}
The first term is handled thanks to a regular Gaussian tail
inequality. Notice indeed that 
\begin{align*}
\Pb(\abs{I} &\geq \delta \vert X_1, \ldots, X_n ) \ind{\rho_\jk^2 \leq
  \mu_{\max} + 2^{j\frac d2} \delta }\\
&\leq 1 \wedge \left\{ \frac{2 \rho_\jk \sigma}{\delta \sqrt{2\pi n}} \exp\left(-
    \frac{n \delta^2}{\rho_\jk^2\sigma^2}  \right)  \right\}\ind{\rho_\jk^2 \leq
  \mu_{\max} + 2^{j\frac d2} \delta} \\
&\leq  1 \wedge \left\{ \frac{2\sigma (\mu_{\max} + 2^{j\frac d2}
    \delta)^{\frac 12}}{\delta \sqrt{2\pi n}} \exp\left(-
    \frac{n \delta^2 \sigma^{-2}}{\mu_{\max} + 2^{j\frac d2} \delta}  \right)
\right \}.
\end{align*}
In addition, notice that $\Exp \phi_\jk(X)^4\inds_{\Hs}(X) \leq \mu_{\max}2^{jd}$ and
$\abs{\phi_\jk(X)^2\inds_{\Hs}(X_i) - \Exp \phi_\jk(X)^2\inds_{\Hs}(X) } \leq 2^{jd +1}$, so that a direct application of
Bernstein inequality leads to
\begin{align*}
\Pb(\Cd_\jk(2^{j\frac d2}\delta)^c) &\leq 2 \exp\left( -\frac{n
    2^{jd}\delta^2}{2^{jd}( 2\mu_{\max}  +
    \frac 43 2^{j\frac d2}\delta )}  \right) = 2 \exp\left( -\frac{n \delta^2}{ 2\mu_{\max}  +
    \frac 43 2^{j\frac d2}\delta }  \right),
\end{align*}
which concludes the proof.
\end{proof}

\begin{proposition}\label{prop:loclinind}
Let $\m$ be a constant such that $\m >0$ and fix $z\in \R^d$ such that
$z \in \linebreak[2]\Bo_{\infty}(2^{-1}, \m)$. Write $\fdel:= \{k
\in \Zr^d: 2^{-1} \in \Supp \phi_k\}$, the set of indexes
corresponding to the scaling functions whose support $\Supp \phi_k$ contains
the point $2^{-1} \in \R^d$. The scaling functions $(\phi_k)$ verify the \textbf{local linear
  independence property} in the sense that $\sum_{k \in \fdel}
\alpha_k \phi_k = 0$ on the domain $\Bo_{\infty}(z,\m)$ if and only if
$\alpha_k=0$ for all $k \in \fdel$. 
\end{proposition}
\begin{proof}
This result is derived from \mycite{Malgouyres1991} and its proof can be
found in \mycite{Monnier2011}.
\end{proof}

\subsection{Proof of the upper-bound results under \Hp{CS2}}\label{sec:proofCS2}

Recall that under \Hp{CS2}, we work with a sample of size $2n$ split
into two sub-samples denoted by $\D_n$ and $\D_n'$. As detailed previously,
similar results as the ones described in \ref{sec:res},
\ref{sec:refinment} and \ref{sec:usefulprop} are still valid with
$\eta^{\maltese}$ under \Hp{CS2}. They in fact all stem from
\ref{th:mainbis}. The proofs remain for the most
part unchanged, with $\J_n$ redefined as $\J_n = \{j_{s},
j_{s}+1, \ldots, J-1, J\}$ where $2^{j_{s}} = \floor{n^{\frac
    1{2s+d}}}$, $\eta^{\maltese}$ in place of
$\eta^@$, $\tilde X_i$ in place of $X_i$ (where we have written
$\tilde u = u - X_{i_x}' + 2^{-j-1}$), and $\Hs_0$
in place of $\Hs$. The sole differences appear in the proofs of
\ref{th:mainbis} and \ref{prop:support}. Let us start
with the proof of \ref{th:mainbis}. 
\begin{proof}[Proof of \ref{th:mainbis}]
Assume we are under \Hp{CS2} and want to control the probability of
deviation of $\eta_j^{\maltese}(x)$ from $\eta(x)$ at a point
$x\in\A$, for some $j \in \J_n$. Recall that $\Hs_0(x)$ stands for the
cell $\Hs_0 = 2^{-j}[0,1]^d$ centered in $x$ at level $j$, that is $\Hs_0(x) = x - 2^{-j-1} +
2^{-j}[0,1]^d$ and denote by $\Os_x$ the event 
\begin{align*}
\Os_x = \{ \#\{i : X_{i}' \in \Hs_0(x)\} \geq 1 \}.
\end{align*}
We can write
\begin{align*}
\Pb( \abs{\eta(x) - \eta_j^{\maltese}(x)}\geq \delta) &= \Pb(
\abs{\eta(x) - \eta_j^{\maltese}(x)}\geq \delta, \Os_x)\\
&+ \Pb( \abs{\eta(x) - \eta_j^{\maltese}(x)}\geq \delta, \Os_x^c).
\end{align*}
Focus first on what happens on the event $\Os_x^c$. The last term can
be controlled easily since the probability that no
single design point $X_i'$ of $\D_n'$ belongs to $\Hs_0(x)$ decreases
exponentially fast with $n$. Notice indeed that, under \Hp{CS2},  
\begin{align*}
\Pb(\Os_x^c) &= (\Pb(X_1' \notin \Hs_0(x)))^n \\
&= (1 - \Pb(X_1' \in \Hs_0(x)))^n\\
&= \left(1 - \int_{\A \cap \Hs_0(x)} \mu(w) dw \right)^n\\
&\leq \left(1 - \mu_{\min}2^{-jd} \lambda\left(2^j(\A-x) \cap [-2^{-1}, 2^{-1}]^d\right) \right)^n\\
&\leq (1 - \mu_{\min}2^{-jd} \min(2\m_0, 2^{-1})^d)^n\\
&\leq \exp( - \mu_{\min} \min(2\m_0, 2^{-1})^d n2^{-jd} ) ,
\end{align*}
where the before last inequality is a direct consequence of
\Hp{S2} and the last one comes from the fact that for any $x \in
[0,1)$, $\ln(1-x) \leq -x$. Now, recall that
$\eta_j^{\maltese}(x) = 0$ on $\Os_x^c$ and $\abs{\eta(x)}\leq M$ since
$\eta \in \Lip^s(\R^d,M)$. So that we obtain
\begin{align*}
\Pb( \abs{\eta(x) - \eta_j^{\maltese}(x)}\geq \delta, \Os_x^c)  &\leq
\exp( - \mu_{\min} \min(2\m_0, 2^{-1})^d n2^{-jd} ) \ind{\delta \leq M},
\end{align*}
which is smaller than the first term in the upper-bound of \ref{th:mainbis}.
Now focus on what happens on the event $\Os_x$. We can write 
\begin{align*}
\Pb(
\abs{\eta(x) - \eta_j^{\maltese}(x)}\geq \delta, \Os_x) &= \Pb(\Os_x) \Exp[ \Pb(
\abs{\eta(x) - \eta_j^{\maltese}(x)}\geq \delta \vert X_{i_x}') \vert
\Os_x]\\
&\leq \Exp[ \Pb( \abs{\eta(x) - \eta_j^{\maltese}(x)}\geq \delta \vert X_{i_x}') \vert \Os_x].
\end{align*}
Therefore, it is enough to control the probability of
deviation of $\eta_j^{\maltese}(x)$ from $\eta(x)$ on $\Os_x$, conditionally on
$X_{i_x}'$. It is controlled in exactly the same way as the probability
of deviation of $\eta^@_j(x)$ from $\eta(x)$ under
\Hp{CS1}, except that we now work with conditional probabilities and
expectations with respect to $X_{i_x}'$. Interestingly, the random
variable $X_{i_x}'$ is independent of the points of $\D_n$ since it is built
upon the design points $(X'_i)$ of $\D_n'$ which are themselves
independent of the points of
$\D_n$. This is a key feature that makes theoretical computations
tractable under \Hp{CS2} and allows to handle $\eta^{\maltese}$ in a
similar way as $\eta^@$ under \Hp{CS1}. As announced above,
\ref{prop:support} is the sole result that is not
obviously true under \Hp{CS2}. However it can be
extended to setting \Hp{CS2} without much trouble (see below). Ultimately, this
proves that, on the event $\Os_x$ and conditionally on $X_{i_x}'$, the
probability of deviation of $\eta_j^{\maltese}(x)$ from $\eta(x)$
verifies \ref{th:main}. So that finally, it
remains to put everything together to obtain the results announced in
\ref{th:mainbis}, which concludes the proof.
\end{proof}

As detailed in \mycite{Monnier2011}, the proof of \ref{prop:support}
can be extended to setting \Hp{CS2}, thanks to the
\textbf{local linear independence property} of the scaling functions
(see \ref{prop:loclinind}) and a compactness argument. In particular,
we obtain the following result, which is proved in \mycite{Monnier2011}.

\begin{lemma}\label{lem:lower}
Let $r \in \N$. Let $\phi$ be the Daubechies' scaling
function of regularity $r$ and $\fdel = \{\nu \in \Zr^d: 2^{-1}
\in \Supp \phi_{\nu} \}$. Then, there exists a strictly positive
absolute constant $c_{\min}$ such that 
\begin{align}
\inf_{u \in \Sp^{\rr^d-1}} \inf_{\m \geq \m_0} \inf_{ z \in \Bo_\infty( 2^{-1}, \m) }
\int_{\Bo_\infty(z , \m) \cap [0,1]^d} \left(\sum_{\nu \in \fdel} u_{\nu}
  \phi_\nu(w ) \right)^2 dw & \geq c_{\min},\label{eq:loweri}\\
\inf_{\m \geq \m_0}\inf_{ z \in \Bo_\infty( 2^{-1}, \m) }
\int_{\Bo_\infty(z , \m)} \sum_{\nu \in \fdel} \phi_\nu(w )^2 dw
&\geq c_{\min}.\label{eq:lowerii}
\end{align}
\end{lemma}

\section*{Appendix}

\subsection*{Generalized Lipschitz spaces}

Here, we sum up relevant facts about Lipschitz and Besov spaces on $\R^d$ as stated in
\mycite[Chap.~3]{Cohen2003} for any $d \in \N$ and \mycite[Chap.~2, \textsection
9]{DeVore1993} for $d=1$. Let us denote by $\Cs(\R^d)$
and $\tilde \Cs(\R^d)$ the spaces of continuous and absolutely continuous
functions on $\R^d$, respectively. Let us denote by $\norm{.}$ the
Euclidean norm of $\R^d$, $f$ a function
defined on $\R^d$ and write $\Delta_h^1(f,x) = \abs{f(x+h) - f(x)}$ for any
$x \in \R^d$. For any
$r\in \N$ and all $x \in \R^d$, we further define the
$r^{th}$-finite difference by induction as follows,
\begin{align*}
\Delta_h^r(f, x) = \Delta_h^1( \Delta_h^{r-1}(f, x)), 
\end{align*}
and the $r^{th}$-modulus of smoothness of $f \in \Cs(\R^d)$ as follows
\begin{align*}
\omega_r(f,t)_{\infty} = \sup_{0 \leq \norm{h} \leq t}
\normL{\Delta_h^r(f,.)}{\Lp_{\infty}(\R^d, \lambda)}.
\end{align*}
Write $s >0$ and $r = \floor{s} +1$. The Besov space
$B^s_{\infty,\infty}$ on $\R^d$, also known
as the generalized Lipschitz space $\Lip^s(\R^d)$, is the
collection of all functions $f \in \tilde \Cs(\R^d) \cap
\Lp_\infty(\R^d, \lambda)$ such that the
semi-norm
\begin{align*}
\abs{f}_{\Lip^s(\R^d)} := \sup_{t>0}
\left(t^{-s}\omega_r(f,t)_{\infty} \right),
\end{align*}
is finite. The norm for $\Lip^s(\R^d)$ is subsequently
defined as 
\begin{align*}
\normL{f}{\Lip^s(\R^d)} := \normL{f}{\Lp_\infty(\R^d,\lambda)} + \abs{f}_{\Lip^s(\R^d)}.
\end{align*}
Fix a real number $M >0$. Throughout the paper, $\Lip^s(\R^d,M)$
refers to the ball of $\Lip^s(\R^d)$ of radius $M$. Obviously, the
elements of $\Lip^s(\R^d,M)$ are $\lambda$-a.e. uniformly bounded by $M$
on $\R^d$.\\
As described in \mycite{DeVore1993, Cohen2003}, there exists an alternative
definition of Lipschitz spaces $\Cs^s(\R^d)$, also known as H\"older
spaces, which goes as follows. For any integer $d$, multi-index $q = (q_1,
\ldots, q_d) \in \N^d$ and $x = (x_1, \ldots, x_d) \in \R^d$, we
define the differential operator $\partial^q$ as usual by $\partial^q
:= \tfrac{\partial^{q_1+ \ldots +q_d}}{\partial^{q_1} x_1
  \ldots \partial^{q_d}x_d}$. For any positive integer $s$, $\Cs^s(\R^d)$
consists of the functions $f$ on $\R^d$ such that $\partial^{q}f$
is bounded and absolutely continuous on $\R^d$, for all $q \in \N^d$
such that $ \abs{q}_1 := q_1 + \ldots +
q_d \leq s$. This definition is extended to non-integer $s$ as
follows,
\begin{align*}
\Cs^s(\R^d) &:= \{ f \in \tilde \Cs(\R^d) \cap\Lp_{\infty}(\R^d, \lambda):
\sup_{x \in \R^d} \Delta_h^1(f,x) \leq C \abs{h}^s \}, \quad 0<s<1,
\\
\Cs^s(\R^d) &:= \{ f \in \tilde \Cs(\R^d) \cap\Lp_{\infty}(\R^d, \lambda):\\
&\qquad \partial^{q}f \in \Cs^{s-m}(\R^d), \abs{q}_1 = m \}, \quad
m<s<m+1, \ m \in \N.
\end{align*} 
It can be shown that, for all non-integer $s>0$, $\Cs^s(\R^d) =
\Lip^s(\R^d)$, while $\Cs^s(\R^d)$ is a strict subset of $\Lip^s(\R^d)$ when $s
\in \N$ (see \mycite[p.~52]{DeVore1993} for examples of functions that
belong to $\Lip^1([0,1])$ but not to $\Cs^1([0,1])$ in the particular case
where $d=1$).\\
Furthermore, we define these function spaces on the subset $\Xs$ of $\R^d$ as
the restriction of their elements to $\Xs$. As explained in
\mycite[Remark~3.2.4]{Cohen2003}, function spaces on $\Xs$ can be
defined by restriction or, alternatively, in an intrinsic way, and both definitions coincide for
fairly general domains $\Xs$ of $\R^d$.\\
Looking at function spaces on $\Xs$ as function
spaces on $\R^d$ restricted to $\Xs$ justifies the use of MRAs of $\Lp_2(\R^d,\lambda)$ in
our local analysis.

\subsection*{MRAs and smoothness analysis}

Multivariate MRAs will always be assumed to be obtained
from a tensorial product of one-dimensional MRAs, as described in
\mycite[\textsection 1.4,~eq.~(1.4.10)]{Cohen2003}. We will denote by
$\phi_\jk(.) = 2^{jd/2} \phi( 2^j . - k)$ the
translated and dilated version of $\phi$ with $k \in \Zr^d$. As usual, we
write $\V_j$ to mean $\Cl(\Span\{\phi_\jk , k\in \Zr^d\})$,
so that $\Cl(\cup_{j \geq 0} \V_j) = \Lp_2(\R^d,\lambda)$ (where the closures are taken
with respect to the $\Lp_2(\R^d,\lambda)$-metric).\\ 
The $r$-MRAs defined in \ref{subsec:polrep} are intimately connected with generalized
Lipschitz spaces. Assume we are given a $r$-MRA with $r \in \N$ and $\eta \in
\Lip^s(\R^d, M)$, where $s \in (0,r)$ and $M>0$. Denote by $\pr_j \eta$ the
orthogonal projection of $\eta$ onto $\V_j$ and by $\re_j\eta = \eta - \pr_j \eta$
the corresponding remainder. Then, we have for all $x \in \R^d$,
$\eta(x) = \pr_j \eta(x) + \re_j\eta(x)$ where $\normL{\re_j \eta }{\Lp_\infty(\R^d, \lambda)} \leq M 2^{-js}$, as
detailed in \mycite[Corollary~3.3.1]{Cohen2003}. It is noteworthy that
the above approximation results remain valid in the particular case
where we work on the subset $\Xs$ of $\R^d$ and consider $\eta$ to be
the restriction to $\Xs$ of an element of $\Lip^s(\R^d)$.

\begin{acknowledgement}
The author would like to thank Dominique Picard for many fruitful
discussions and suggestions. He would also like to acknowledge
interesting conversations with G\'{e}rard Kerkyacharian. Finally, he
would like to thank two anonymous referees and an associate editor
whose constructive comments led to a full refactoring of the paper and
largely contributed to improve it. 
\end{acknowledgement}

\end{document}